\documentclass[journal]{IEEEtran}
\ifCLASSINFOpdf
\usepackage[pdftex]{graphicx}
\else
\fi
\usepackage{amsmath, amssymb, amsthm}
\usepackage{amsfonts}
\usepackage{algorithmic}
\ifCLASSOPTIONcompsoc
\usepackage[caption=false,font=normalsize,labelfont=sf,textfont=sf]{subfig}
\else
\usepackage[caption=false,font=footnotesize]{subfig}
\fi

\usepackage{color}
\usepackage{xcolor}
\usepackage{tabularx}
\usepackage{multirow}

\newcommand{\R}{\mathbb{R}}
\newcommand{\Gra}{\mathcal{G}}

\newcommand{\bx}{\mathbf{x}}
\newcommand{\by}{\mathbf{y}}

\newcommand{\barf}{f}

\newcommand{\df}{\nabla F}

\newcommand{\trace}{\mathrm{trace}}
\newcommand{\mR}{\mathbf{R}}
\newcommand{\mC}{\mathbf{C}}
\newcommand{\mA}{\mathbf{A}}
\newcommand{\mrR}{\mathrm{R}}
\newcommand{\mrC}{\mathrm{C}}
\newcommand{\mrE}{\mathrm{E}}
\newcommand{\mrT}{\mathrm{T}}
\newcommand{\mrS}{\mathrm{S}}
\newcommand{\mrD}{\mathrm{D}}
\newcommand{\mrQ}{\mathrm{Q}}
\newcommand{\mT}{\mathbf{T}}
\newcommand{\mD}{\mathbf{D}}
\newcommand{\mE}{\mathbf{E}}
\newcommand{\bE}{\mathbb{E}}
\newcommand{\mI}{\mathbf{I}}
\newcommand{\mV}{\mathbf{V}}
\newcommand{\balpha}{\boldsymbol{\alpha}}
\newcommand{\T}{\intercal}

\newcommand{\spa}[1]{\mathrm{span}\{#1\}}
\newcommand{\nul}[1]{\mathrm{null}\{#1\}}
\newcommand{\dia}[1]{\mathrm{diag}\left\{#1\right\}}

\newcommand{\inneighbor}[1]{\mathcal{N}^\mathrm{in}_{#1}}
\newcommand{\outneighbor}[1]{\mathcal{N}^\mathrm{out}_{#1}}
\newcommand{\mB}{\mathbf{B}}
\newcommand{\mM}{\mathbf{M}}
\newcommand{\mW}{\mathbf{W}}
\newcommand{\mS}{\mathbf{S}}
\newcommand{\mJ}{\mathbf{J}}
\newcommand{\mQ}{\mathbf{Q}}
\renewcommand{\trace}{\mathbf{tr}}

\definecolor{darkgreen}{rgb}{0.01, 0.75,0.24}

\newcommand{\mx}{\mathbf{x}}
\newcommand{\my}{\mathbf{y}}
\newcommand{\ox}{\bar{x}}
\newcommand{\oy}{\bar{y}}

\newtheorem{definition}{Definition}
\newtheorem{assumption}{Assumption}
\newtheorem{lemma}{Lemma}
\newtheorem{remark}{Remark}
\newtheorem{theorem}{Theorem}

\begin{document}
	%
	\title{Push-Pull Gradient Methods for Distributed Optimization in Networks}
	%
	%
	%
	
	\author{Shi~Pu,
		Wei~Shi, Jinming~Xu,
		and~Angelia~Nedi{\'c}
		\thanks{The first author and the second author contribute equally. Parts of the results appear in Proceedings of the 57th IEEE Conference on Decision and Control \cite{pu2018push}. (Corresponding authors: Shi Pu; Jinming Xu.)}
		\thanks{This work was supported in parts by the NSF grant CCF-1717391, the ONR grant no.\ N000141612245, and the SRIBD Startup Fund JCYJ-SP2019090001.}
		\thanks{S. Pu is with Institute for Data and Decision Analytics, The Chinese University of Hong Kong, Shenzhen, China and Shenzhen Research Institute of Big Data.  (e-mail: pushi@cuhk.edu.cn).}
		\thanks{W. Shi was with the Department of Electrical Engineering, Princeton University, Princeton, NJ 08854 USA (e-mail: Wilbur.Shi@princeton.edu).}
		\thanks{J. Xu is with the College of Control Science and Engineering, Zhejiang University, China (e-mail: jimmyxu@zju.edu.cn).}
		\thanks{A. Nedi{\'c} is with the School of Electrical, Computer, and Energy Engineering, Arizona
			State University, Tempe, AZ 85281 USA (e-mail: Angelia.Nedich@asu.edu).}
	}

	\maketitle
	
	\begin{abstract}
		In this paper, we focus on solving a distributed convex optimization problem in a network, where each agent has its own convex cost function and the goal is to minimize the sum of the agents' cost functions while obeying the network connectivity structure. In order to minimize the sum of the cost functions,
		we consider new distributed gradient-based methods where each node maintains two estimates, namely, an estimate of the optimal decision variable and an estimate of the gradient
		for the average of the agents' objective functions. 
		From the viewpoint of an agent, the information about the gradients is pushed to the neighbors, while the information about the decision variable is pulled from 
		the neighbors hence giving the name ``push-pull gradient methods". 
		The methods utilize two different graphs for the information exchange among agents, and as such, unify the algorithms with different types of distributed architecture, including decentralized (peer-to-peer), centralized (master-slave), and semi-centralized (leader-follower) architecture. We show that the proposed algorithms and their many variants converge linearly for strongly convex and smooth objective functions over a network (possibly with unidirectional data links) in both synchronous and asynchronous random-gossip settings. In particular, under the random-gossip setting, ``push-pull'' is the first class of algorithms for distributed optimization over directed graphs. Moreover, we numerically evaluate our proposed algorithms in both scenarios, and show that they outperform other existing linearly convergent schemes, especially for ill-conditioned problems and networks that are not well balanced.
	\end{abstract}
	

	%
	\IEEEpeerreviewmaketitle

	\section{Introduction}
	%
	%
	%
	%
	\IEEEPARstart{I}{n} this paper, we consider a system involving $n$ agents whose goal
	is to collaboratively solve the following problem:
	\begin{equation} \label{problem}
	\begin{array}{c}
	\min\limits_{x\in\R^p}~\barf(x):=\sum\limits_{i=1}^n f_i(x),
	\end{array}
	\end{equation}
	where $x$ is the global decision variable and each function $f_i: \R^p\rightarrow \R$ is convex and known by agent $i$ only.
	The agents are embedded in a communication network, and their
	goal is to obtain an optimal and consensual solution through local neighbor communications and information exchange. This local exchange is desirable in situations where 
	the exchange of a large amount of data is prohibitively expensive due to limited communication resources.
	
	To solve problem~\eqref{problem} in a networked system of $n$ agents, many algorithms have been proposed under various assumptions on the objective functions and on the underlying networks/graphs. Static undirected graphs were extensively considered in the literature \cite{Shi2014,Shi2015,olshevsky2017linear,qu2017harnessing,scaman2017optimal}. 
	References \cite{nedic2009distributed,xu2017convergence,nedic2017achieving} studied time-varying and/or stochastic undirected networks.
	Directed graphs were discussed in~\cite{nedic2015distributed,nedic2016stochastic,zeng2017extrapush,nedic2017achieving,xi2017dextra,xi2018linear}.
	Centralized (master-slave) algorithms were discussed in~\cite{boyd2011distributed}, where extensive applications in learning can be found. Parallel, coordinated, and asynchronous algorithms were discussed in~\cite{peng2016arock} and the references therein. The reader is also referred to the recent paper~\cite{nedic2018network} and the references therein for a comprehensive survey on distributed optimization algorithms.
	
	In the first part of this paper, we introduce a novel gradient-based algorithm (Push-Pull) for distributed (consensus-based) optimization in directed graphs. Unlike the push-sum type protocol used in the previous literature~\cite{nedic2017achieving,xi2018linear}, our algorithm uses a row stochastic matrix for the mixing of the decision variables, while it employs a column stochastic matrix for tracking the average gradients. Although motivated by a fully decentralized scheme, we show that Push-Pull can work both in fully decentralized networks and in two-tier networks. 
	
	Gossip-based communication protocols are popular choices for distributed computation due to their low communication costs \cite{boyd2006randomized,lu2011gossip,lee2016asynchronous,mathkar2016nonlinear}.
	In the second part of this paper, we consider a random-gossip push-pull algorithm (G-Push-Pull) where at each iteration, an agent wakes up uniformly randomly and communicates with one or two of its neighbors. 	
	Both Push-Pull and G-Push-Pull have different variants. We show that they all converge linearly to the optimal solution for strongly convex and smooth objective functions.
	
	\subsection{Related Work}
	Our emphasis in the literature review is on the decentralized optimization, since our approach builds on a new understanding of the decentralized consensus-based methods 
	for directed communication networks. Most references, including \cite{Shi2014,Shi2015,mokhtari2016decentralized,varagnolo2016newton,olshevsky2017linear,scaman2017optimal,uribe2017optimal,xu2015augmented,nedic2017geometrically,di2016next,Shi2015_2,li2019decentralized}, often restrict the underlying network connectivity structure, or more commonly require doubly stochastic mixing matrices.
	The work in~\cite{Shi2014} has been the first to demonstrate the linear convergence of an ADMM-based 
	decentralized optimization scheme. Reference~\cite{Shi2015} uses a gradient difference structure in the algorithm to provide the first-order decentralized optimization algorithm which is capable of achieving the typical convergence rates of a centralized gradient method, while 
	references~\cite{mokhtari2016decentralized,varagnolo2016newton} deal with the second-order decentralized methods. 
	By using Nesterov's acceleration, reference~\cite{olshevsky2017linear} has obtained a method whose convergence time scales linearly in the number of agents $n$, which is the best scaling with $n$ currently known.
	More recently, for a class of so-termed dual friendly functions, 
	papers~\cite{scaman2017optimal,uribe2017optimal} have obtained an optimal decentralized consensus optimization algorithm whose dependency on the condition number\footnote{The condition number of a smooth and strongly convex function is the ratio of its gradient Lipschitz constant and its strong convexity constant.} of the system's objective function achieves the best known scaling in the order of $O(\sqrt{\kappa})$.
	Work in~\cite{Shi2015_2,li2019decentralized} investigates proximal-gradient methods which can tackle \eqref{problem} with proximal friendly component functions. Paper~\cite{wu2016decentralized} extends
	the work in~\cite{Shi2014} to handle asynchrony and delays. 
	References~\cite{pu2018swarming,pu2018distributed} considered a stochastic variant of problem~\eqref{problem} 
	in asynchronous networks. A tracking technique has been recently employed to develop
	decentralized algorithms for tracking the average of the Hessian/gradient in second-order methods~\cite{varagnolo2016newton}, allowing uncoordinated stepsize~\cite{xu2015augmented,nedic2017geometrically}, handling non-convexity \cite{di2016next}, and achieving linear convergence over 
	time-varying graphs~\cite{nedic2017achieving}.
	
	For directed graphs, to eliminate the need of constructing a doubly stochastic matrix in reaching consensus\footnote{Constructing a doubly stochastic matrix over a directed graph needs weight balancing which requires an independent iterative procedure across the network; consensus is a basic coordination technique in decentralized optimization.}, reference \cite{Kempe2003} proposes the push-sum protocol. Reference \cite{tsianos2012push} has been the first to propose a push-sum based distributed optimization algorithm for directed graphs. Then, based on the push-sum technique again, a decentralized subgradient method for time-varying directed graphs has been proposed and analyzed in~\cite{nedic2015distributed}. Aiming to improve convergence for a smooth objective function and a fixed directed graph, the work in~\cite{xi2017dextra,zeng2017extrapush} modifies the algorithm from~\cite{Shi2015} with the push-sum technique, thus providing a new algorithm which converges linearly for a strongly convex objective function on a static graph. However, the algorithm requires a careful selection of the stepsize which may be even non-existent in some cases~\cite{xi2017dextra}. This stability issue has been resolved in~\cite{nedic2017achieving} in a more general setting of time-varying directed graphs. The work in~\cite{xi2018linear,xin2019frost} considers an algorithm that uses only row-stochastic mixing matrices and still achieves linear convergence over fixed directed graphs.
	
	Simultaneously and independently, a paper \cite{xin2018linear} has proposed an algorithm that is similar to the synchronous variant proposed in this paper. By contrast, the work in \cite{xin2018linear} does not show that the algorithm unifies different architectures. Moreover, asynchronous or time-varying cases were not discussed either therein.
	
	\subsection{Main Contribution}
	The main contribution of this paper is threefold. First, we design new distributed optimization methods (Push-Pull and G-Push-Pull) and their many variants for directed graphs. These methods utilize two different graphs for the information exchange among agents, and as such, unify different computation and communication architectures, including decentralized (peer-to-peer), centralized (master-slave), and semi-centralized (leader-follower) architecture. To the best of our knowledge, these are the first algorithms in the literature that enjoy such property.
	
	Second, we establish the linear convergence of the proposed methods in both synchronous (Push-Pull) and asynchronous random-gossip (G-Push-Pull) settings. In particular, G-Push-Pull is the first class of gossip-type algorithms for distributed optimization over directed graphs.
	
	Finally, in our proposed methods each agent in the network is allowed to use a different nonnegative stepsize, and only one of such stepsizes needs to be positive. This is a unique feature compared to the existing literature (e.g., \cite{nedic2017achieving,xi2018linear}).
	
	Some of the results related to a variant of Push-Pull will appear in Proceedings of the 57th IEEE Conference on Decision and Control \cite{pu2018push}. In contrast, the current work analyzes a different, more communication-efficient variant of Push-Pull, adopts an uncoordinated stepsize policy which generalizes the scheme in \cite{pu2018push}, and introduces G-Push-Pull in extra. It also contains detailed proofs omitted from the conference version.
	
	\subsection{Organization of the Paper}
	The structure of this paper is as follows. We first provide notation and state basic assumptions in 
	Subsection~\ref{sec: problem}. Then we introduce the push-pull gradient method in Section~\ref{sec: Push-Pull} along with the intuition of its design and some examples explaining how it relates to (semi-)centralized and decentralized optimization. 
	We establish the linear convergence of the push-pull algorithm in Section~\ref{sec: conv_analysis}. In Section~\ref{sec: G-Push-Pull} we introduce the random-gossip push-pull method (G-Push-Pull) and demonstrate its linear convergence in Section~\ref{sec: conv_analysis_G-Push-Pull}.
	In Section~\ref{sec: simulation} we conduct numerical experiments to verify our theoretical claims. 
	Concluding remarks are given in Section \ref{sec: conclusion}.
	
	
	\subsection{Notation and Assumption}
	\label{sec: problem}
	Throughout the paper, vectors default to columns if not otherwise specified.
	Let $\mathcal{N}=\{1,2,\ldots,n\}$ be the set of agents. Each agent $i\in\mathcal{N}$ holds a local copy $x_i\in\mathbb{R}^p$ of the decision variable and an auxiliary variable $y_i\in\mathbb{R}^p$ tracking the average gradients, where their
	values at iteration $k$ are denoted by $x_{i,k}$ and $y_{i,k}$, respectively. Let
	\begin{eqnarray*}
		&\mx:=[x_1, x_2, \ldots,x_n]^{\T}\in\mathbb{R}^{n\times p},\\
		&\my:=[y_1, y_2, \ldots,y_n]^{\T}\in\mathbb{R}^{n\times p}.
	\end{eqnarray*}
	Define $F(\mx)$ to be an aggregate objective function of the local variables,
	i.e., $F(\mx):=\sum_{i=1}^nf_i(x_i)$, and write
	\begin{equation*}
	\nabla F(\mx):=\left[\nabla f_1(x_1)^{\T}, \nabla f_2(x_2)^{\T}, \ldots, \nabla f_n(x_n)^{\T}\right]\in\mathbb{R}^{n\times p}.
	\end{equation*}
	We use the symbol $\trace\{\cdot\}$ to denote the trace of a square matrix.
	
	\begin{definition}\label{def: norm n p}
		Given an arbitrary vector norm $\|\cdot\|$ on $\mathbb{R}^n$, for any $\mx\in\mathbb{R}^{n\times p}$, we define
		\begin{equation*}
		\|\mx\|:=\left\|\left[\|\mx^{(1)}\|,\|\mx^{(2)}\|,\ldots,\|\mx^{(p)}\|\right]\right\|_2,
		\end{equation*}
		where $\mx^{(1)},\mx^{(2)},\ldots,\mx^{(p)}\in\mathbb{R}^{n}$ are columns of $\mx$, and $\|\cdot\|_2$ represents the $2$-norm.
	\end{definition}
	
	We make the following assumption on the functions $f_i$ in~\eqref{problem}.
	\begin{assumption}
		\label{asp; strconvex Lipschitz}
		Each $f_i$ is $\mu$-strongly convex and its gradient is $L$-Lipschitz continuous, i.e., for any $x,x'\in\mathbb{R}^p$,
		\begin{equation*}
		\begin{split}
		& \langle \nabla f_i(x)-\nabla f_i(x'),x-x'\rangle\ge \mu\|x-x'\|_2^2,\\
		& \|\nabla f_i(x)-\nabla f_i(x')\|_2\le L \|x-x'\|_2.
		\end{split}
		\end{equation*}
	\end{assumption}
	Under Assumption \ref{asp; strconvex Lipschitz}, there exists a unique optimal 
	solution $x^*\in\mathbb{R}^{p}$ to problem~\eqref{problem}.	
	
	We use directed graphs to model the interaction topology among agents. A directed graph (digraph) is a pair $\mathcal{G}=(\mathcal{N},\mathcal{E})$, where $\mathcal{N}$ is the set of vertices (nodes) and the edge set $\mathcal{E}\subseteq \mathcal{N}\times \mathcal{N}$ consists of ordered pairs of vertices.	
	If there is a directed edge from node $i$ to node $j$ in $\mathcal{G}$, or $(i,j)\in\mathcal{E}$, then $i$ is defined as the parent node and $j$ is defined as the child node. Information can be transmitted from the parent node to the child node directly.
	A directed path in graph $\mathcal{G}$ is a sequence of edges $(i,j)$, $(j,k)$, $(k,l)\ldots$.
	Graph $\mathcal{G}$ is called strongly connected if there is a directed path between any pair of distinct vertices. 
	A directed tree is a digraph where every vertex, except for the root, has only one parent. 
	A spanning tree of a digraph is a directed tree that connects the root to all other vertices in the graph.
	A subgraph $\mathcal{S}$ of graph $\mathcal{G}$ is a graph whose set of vertices and set of edges are all subsets of $\mathcal{G}$ (see~\cite{godsil2013algebraic}).
	
	Given a nonnegative matrix\footnote{A matrix is nonnegative if all its elements are nonnegative.} $\mathbf{M}=[m_{ij}]\in\mathbb{R}^{n\times n}$, the digraph induced by the matrix $\mathbf{M}$ is 
	denoted by $\mathcal{G}_\mathbf{M}=(\mathcal{N},\mathcal{E}_\mathbf{M})$, where 
	$\mathcal{N}=\{1,2,\ldots,n\}$ and $(j,i)\in\mathcal{E}_\mathbf{M}$ iff (if and only if) $m_{ij}>0$. 
	We let $\mathcal{R}_\mathbf{M}$ be the set of roots of all possible spanning trees in the graph $\mathcal{G}_\mathbf{M}$. For an arbitrary agent $i\in\mathcal{N}$, we define its in-neighbor set $\inneighbor{\mathbf{M},i}$ as the collection of all individual agents that $i$ can actively and reliably pull data from; we also define its out-neighbor set $\outneighbor{\mathbf{M},i}$ as the collection of all individual agents that can passively and reliably receive data from $i$. In the situation when the set is time-varying, we further add a subscript to indicate it generates a sequence of sets. For example, $\inneighbor{\mathbf{M},i,k}$ is the in-neighbor set of $i$ at time/iteration $k$. 
	
	\section{A Push-Pull Gradient Method}
	\label{sec: Push-Pull}
	To proceed, we first illustrate and highlight the proposed algorithm, which we call Push-Pull in the following (Algorithm 1).
	
	\smallskip	
	\begin{table}
		\normalsize
		\centering
		{\textbf{Algorithm 1: Push-Pull}}
		
		\smallskip
		\begin{tabularx}{0.5\textwidth}{X}
			\hline
			Each agent $i$ chooses its local step size $\alpha_i\ge 0$,\\
			\quad in-bound mixing/pulling weights $R_{ij}\geq0$ for all $j\in\inneighbor{\mathbf{R},i}$,\\
			\quad and out-bound pushing weights $C_{li}\geq0$ for all $l\in\outneighbor{\mathbf{C},i}$;\\
			Each agent $i$ initializes with any arbitrary $x_{i,0}\in\R^p$ and $y_{i,0}=\nabla f_i(x_{i,0})$;\\
			\textbf{for} $k=0,1,\cdots$, \textbf{do} \\
			\qquad for each $i\in\mathcal{N}$,\\
			\qquad agent $i$ pulls $(x_{j,k}-\alpha_j y_{j,k})$ from each $j\in\inneighbor{\mathbf{R},i}$;\\
			\qquad agent $i$ pushes $C_{li}y_{i,k}$ to each $l\in\outneighbor{\mathbf{C},i}$;\\
			\qquad for each $i\in\mathcal{N}$,\\
			\qquad \quad $x_{i,k+1} =  \sum_{j=1}^n R_{ij}(x_{j,k}-\alpha_j y_{j,k})$;\\
			\qquad \quad $y_{i,k+1} =  \sum_{j=1}^n C_{ij}y_{j,k}+\nabla f_i(x_{i,k+1})-\nabla f_i(x_{i,k})$;\\
			\textbf{end for}\\
			\hline
		\end{tabularx}
	\end{table}
	\smallskip
	
	Algorithm 1 (Push-Pull) can be rewritten in the following aggregated form:
	\begin{subequations}\label{algorithm P-P}
		\begin{align}
		&\mx_{k+1} =  \mathbf{R}(\mx_k-\boldsymbol{\alpha} \my_k),\label{eq:x-update}\\
		&\my_{k+1} =  \mathbf{C}\my_k+\nabla F(\mx_{k+1})-\nabla F(\mx_k),\label{eq:y-update}
		\end{align}
	\end{subequations}
	where $\boldsymbol{\alpha}=\text{diag}\{\alpha_1,\alpha_2,\ldots,\alpha_n\}$ is a nonnegative diagonal matrix and $\mathbf{R}=[R_{ij}],\,\mathbf{C}=[C_{ij}]\in\mathbb{R}^{n\times n}$. We make the following assumption on the matrices $\mathbf{R}$ and $\mathbf{C}$. 
	\begin{assumption}
		\label{asp: stochastic}
		The matrix $\mathbf{R}\in\mathbb{R}^{n\times n}$ is nonnegative row-stochastic and $\mathbf{C}\in\mathbb{R}^{n\times n}$ is nonnegative column-stochastic, i.e., $\mathbf{R}\mathbf{1}=\mathbf{1}$ and $\mathbf{1}^{\T} \mathbf{C}=\mathbf{1}^{\T}$. In addition, the diagonal entries of $\mathbf{R}$ and $\mathbf{C}$ are positive, i.e., $R_{ii}>0$ and  $C_{ii}>0$ for all $i\in \mathcal{N}$.
	\end{assumption}
	As a result of $\mathbf{C}$ being column-stochastic,  we have by induction that
	\begin{equation}
	\label{oy_k}
	\frac{1}{n}\mathbf{1}^{\T}\my_k=\frac{1}{n}\mathbf{1}^{\T}\nabla F(\mx_{k}), \qquad \forall k.
	\end{equation}
	Relation (\ref{oy_k}) is critical for (a subset of) the agents to track the average gradient $\mathbf{1}^{\T}\nabla F(\mx_{k})/n$ through the $\my$-update.
	\begin{remark}
		At each iteration, each agent will ``push'' information about gradients to its out-neighbors and ``pull'' the decision variables from its in-neighbors, respectively.
		Each component of $\my_k$ plays the role of tracking the average gradient using the column-stochastic matrix $\mC$ while each component of $\mx_k$ performs optimization seeking by average consensus using a row-stochastic matrix $\mR$.
		The structure of Algorithm 1 resembles that of the gradient tracking methods as 
		in~\cite{nedic2017achieving,xu2015augmented} with the doubly stochastic matrix being split into a row-stochastic matrix and a column-stochastic matrix. 
		Such an asymmetric $\mathbf{R}$-$\mathbf{C}$ structure has already been used in the literature for achieving average 
		consensus~\cite{cai2012average}. 
		However, the proposed optimization algorithm can not be interpreted as a linear system since it introduces nonlinear dynamics due to the gradient terms.
	\end{remark}
	We now give the condition on the structures of  graphs $\mathcal{G}_\mathbf{R}$ and $\mathcal{G}_{\mathbf{C}^{\T}}$ induced by matrices $\mathbf{R}$ and $\mathbf{C}^{\T}$, respectively. Note that $\mathcal{G}_{\mathbf{C}^{\T}}$ is identical to the graph $\mathcal{G}_{\mathbf{C}}$ with all its edges reversed.
	\begin{assumption}
		\label{asp: nonempty root set}
		The graphs $\mathcal{G}_\mathbf{R}$ and $\mathcal{G}_{\mathbf{C}^{\T}}$ each contain at least one spanning tree. Moreover, there exists at least one node that is a root of spanning trees for both $\mathcal{G}_\mathbf{R}$ and $\mathcal{G}_{\mathbf{C}^{\T}}$, i.e., $\mathcal{R}_\mathbf{R}\cap\mathcal{R}_{\mathbf{C}^{\T}}\neq \emptyset$, where $\mathcal{R}_\mathbf{R}$ (resp., $\mathcal{R}_{\mathbf{C}^{\T}}$) is the set of roots of all possible spanning trees in the graph $\mathcal{G}_\mathbf{R}$ (resp., $\mathcal{G}_{\mathbf{C}^{\T}}$).
	\end{assumption}
	
	Assumption \ref{asp: nonempty root set} is weaker than requiring that both $\mathcal{G}_\mathbf{R}$ and $\mathcal{G}_{\mathbf{C}}$ are strongly connected, which was assumed in most previous works (e.g., \cite{nedic2017achieving,xi2018linear,xin2018linear}). This relaxation offers us more flexibility in designing graphs $\mathcal{G}_\mathbf{R}$ and $\mathcal{G}_{\mathbf{C}}$.
	For instance, suppose that we have a strongly connected communication graph $\mathcal{G}$. Then there are multiple ways to construct $\mathcal{G}_\mathbf{R}$ and $\mathcal{G}_{\mathbf{C}}$ satisfying Assumption \ref{asp: nonempty root set}. One trivial approach is to set $\mathcal{G}_\mathbf{R}=\mathcal{G}_{\mathbf{C}}=\mathcal{G}$. Another way is to pick at random $i_r\in\mathcal{N}$ and let $\mathcal{G}_\mathbf{R}$ (resp., $\mathcal{G}_{\mathbf{C}}$) be a spanning tree (resp., reversed spanning tree) contained in $\mathcal{G}$ with $i_r$ as its root.
	Once graphs $\mathcal{G}_\mathbf{R}$ and $\mathcal{G}_{\mathbf{C}}$ are established, matrices $\mathbf{R}$ and $\mathbf{C}$ can be designed accordingly. 
	\begin{remark} 
		There are different ways to design the weights in $\mR$ and $\mC$ by each agent locally.  For example, each agent $i$ may choose $R_{ij}=\frac{1}{|\inneighbor{\mathbf{R},i}|+c_R}$ for some constant $c_R>0$ for all $j\in \inneighbor{\mathbf{R},i}$ and let $R_{ii}=1-\sum_{j\in \inneighbor{\mathbf{R},i}}$. Similarly, agent $i$ can choose
		$C_{li}=\frac{1}{|\outneighbor{\mathbf{C},i}|+c_C}$ for some constant $c_C>0$ for all $l\in\outneighbor{\mathbf{C},i}$ and let $C_{ii}=1-\sum_{l\in \outneighbor{\mathbf{C},i}}$. Such a choice of mixing weights will render $\mR$ row-stochastic and $\mC$ column-stochastic, thus satisfying Assumption \ref{asp: stochastic}.
	\end{remark}
	
	We have the following result from Assumption~\ref{asp: stochastic} and Assumption~\ref{asp: nonempty root set}.
	\begin{lemma}
		\label{lem: eigenvectors u v}
		Under Assumption~\ref{asp: stochastic} and Assumption~\ref{asp: nonempty root set}, the matrix $\mathbf{R}$ has a unique nonnegative left eigenvector $u^{\T}$ (w.r.t.\ eigenvalue $1$) with $u^{\T}\mathbf{1}=n$, and the matrix $\mathbf{C}$ has a unique nonnegative right eigenvector $v$ (w.r.t. eigenvalue $1$) with $\mathbf{1}^{\T}v=n$ (see \cite{horn1990matrix}). Moreover, eigenvector $u^{\T}$ (resp., $v$) is nonzero only on the entries associated with agents $i\in\mathcal{R}_{\mathbf{R}}$ (resp., $j\in\mathcal{R}_{\mathbf{C}^{\T}}$), and $u^{\T}v>0$.
	\end{lemma}
	\begin{proof}
		See Appendix \ref{subsec: proof eigenvectors u v}.
	\end{proof}
	
	Finally, we assume the following condition regarding the step sizes $\{\alpha_i\}$.
	\begin{assumption}
		\label{asp: step sizes}
		There is at least one agent $i\in\mathcal{R}_\mathbf{R}\cap\mathcal{R}_{\mathbf{C}^{\T}}$ whose step size $\alpha_i$ is positive.
	\end{assumption}
	
	Assumption \ref{asp: nonempty root set} and Assumption \ref{asp: step sizes} hint on the crucial role of the set $\mathcal{R}_\mathbf{R}\cap\mathcal{R}_{\mathbf{C}^{\T}}$. In what follows, we provide some intuition for the development of Push-Pull and an interpretation of the algorithm from another perspective. The discussions will shed light on the rationale behind the assumptions.
	
	To motivate the development of Push-Pull, let us consider the optimality condition for \eqref{problem} in the following form:
	\begin{subequations}\label{eq:opt_cond_consensus}
		\begin{align}
		&\mx^*\in\nul{\mathbf{I}-\mathbf{R}},\label{eq:oc_line1_consensus}\\
		&\mathbf{1}^{\T}\df(\mx^*)=0,\label{eq:oc_line2_consensus}
		\end{align}
	\end{subequations}
	where $\mathbf{x}^*:=\mathbf{1}{x^*}^{\T}$ and $\mathbf{R}$ satisfies Assumption \ref{asp: stochastic}.
	Consider the algorithm in~\eqref{algorithm P-P}. Suppose that the algorithm produces two sequences $\{\mx_k\}$ and $\{\my_k\}$ converging to some points $\mx_\infty$ and $\my_\infty$, respectively. Then from (\ref{eq:x-update}) and (\ref{eq:y-update}) we would have 
	\begin{subequations}\label{eq:converge_consensus}
		\begin{align}
		&(\mathbf{I}-\mathbf{R})(\mx_\infty-\boldsymbol{\alpha}\my_\infty) +\boldsymbol{\alpha}\my_\infty=0, \label{eq:converge_line1_consensus}\\
		&(\mathbf{I}-\mathbf{C})\my_\infty=0. \label{eq:converge_line2_consensus}
		\end{align}
	\end{subequations}
	If $\spa{\mathbf{I}-\mathbf{R}}$ and $\boldsymbol{\alpha}\cdot\nul{\mathbf{I}-\mathbf{C}}$ are disjoint\footnote{This is a consequence of Assumption \ref{asp: step sizes} and the relation $u^{\T}v>0$ from Lemma \ref{lem: eigenvectors u v}.}, from (\ref{eq:converge_consensus}) we would have $\mx_\infty\in\nul{\mathbf{I}-\mathbf{R}}$ and $\boldsymbol{\alpha}\my_\infty=\mathbf{0}$. Hence $\mx_\infty$ satisfies the optimality condition in~\eqref{eq:oc_line1_consensus}. 
	In light of (\ref{eq:converge_line2_consensus}), Assumption \ref{asp: step sizes}, and Lemma \ref{lem: eigenvectors u v}, we have $\my_\infty\in\nul{\boldsymbol{\alpha}}\cap\nul{\mathbf{I}-\mathbf{C}}=\{\mathbf{0}\}$.
	Then from (\ref{oy_k}) we know that $\mathbf{1}^{\T}\df(\mx_\infty)=\mathbf{1}^{\T}\my_\infty=\mathbf{0}$, which is exactly the optimality condition in~\eqref{eq:oc_line2_consensus}.
	
	For another interpretation of Push-Pull, notice that under Assumptions \ref{asp: stochastic} and \ref{asp: nonempty root set}, with linear rates of convergence,
	\begin{equation}\label{limit matrices}
	\lim_{k\rightarrow\infty}\mathbf{R}^k=\frac{\mathbf{1}u^{\T}}{n},\ \
	\lim_{k\rightarrow\infty}\mathbf{C}^k=\frac{v\mathbf{1}^{\T}}{n}.
	\end{equation}
	Thus with comparatively small step sizes, relation (\ref{limit matrices}) together with (\ref{oy_k}) implies that $\mx_k\simeq \mathbf{1}u^{\T}\mx_{k-K}/n$ (for some fixed $K>0$) and $\my_k\simeq v\mathbf{1}^{\T}\nabla F(\mx_k)/n$. From the proof of Lemma \ref{lem: eigenvectors u v}, eigenvector $u$ (resp., $v$) is nonzero only on the entries associated with agents $i\in\mathcal{R}_{\mathbf{R}}$ (resp., $j\in\mathcal{R}_{\mathbf{C}^{\T}}$). Hence $\mx_k\simeq \mathbf{1}u^{\T}\mx_{k-K}/n$ indicates that only the state information of agents $i\in\mathcal{R}_{\mathbf{R}}$ are pulled by the entire network, and $\my_k\simeq v\mathbf{1}^{\T}\nabla F(\mx_k)/n$ implies that only agents $j\in\mathcal{R}_{\mathbf{C}^{\T}}$ are pushed and tracking the average gradients. 
	This ``push" and ``pull" information structure gives the name of the algorithm.
	The assumption $\mathcal{R}_\mathbf{R}\cap\mathcal{R}_{\mathbf{C}^{\T}}\neq \emptyset$ essentially says at least one agent needs to be both ``pulled" and ``pushed''.

	The above discussion has mathematically interpreted why the use of row stochastic matrices and column stochastic matrices is reasonable. Now let us explain from the implementation aspect why Algorithm 1 is called ``Push-Pull'' and why it is more feasible to be implemented with ``Push'' and ``Pull'' at the same time.
	When the information across agents need to be diffused/fused, either an agent needs to know what scaling weights it needs to put on the quantities sending out to other agents, or it needs to know how to combine the quantities coming in with correct weights. In particular, we have the following specific weight assignment strategies:
	\begin{itemize}
		\item[A)] For the networked system to maintain $\sum_l C_{li}=1$, an apparently convenient way is to let agent $i$ scale its data by $C_{li},\forall l$, before sending/pushing out messages. In this way, it becomes agent $i$'s responsibility to synchronize out-neighbors' receptions of messages and it is natural to employ a reliable push-communication-protocol to implement such operations. 
		\item[B)] 
		Unlike what happens in A), to maintain $\sum_j R_{ij}=1$, the only seemingly feasible way is to let the receiver $i$ perform the tasks of scaling and combination/addition since it would be difficult for the sender to know the weights or adjust the weights accordingly, especially when the network changes. Thus, it is natural to employ a pull-communication-protocol for the above operations.
	\end{itemize}
	
	\subsection{Unifying Different Distributed Computational Architecture}\label{sec: Push-Pull_unify}
	We now demonstrate how the proposed algorithm (\ref{algorithm P-P}) unifies different types of distributed architecture, including decentralized, centralized, and semi-centralized  architecture.
	For the fully decentralized case, suppose we have a graph $\mathcal{G}$ that is undirected and connected. Then we can set $\mathcal{G}_\mathbf{R}=\mathcal{G}_{\mathbf{C}}=\mathcal{G}$ and let $\mathbf{R}=\mathbf{C}$ be symmetric matrices, in which case the proposed algorithm degrades to the one considered in~\cite{nedic2017achieving,xu2015augmented}; if the graph is directed and strongly connected, we can also let $\mathcal{G}_\mathbf{R}=\mathcal{G}_{\mathbf{C}}=\mathcal{G}$ and design the weights for $\mathbf{R}$ and $\mathbf{C}$ correspondingly.
	
	To illustrate the less straightforward situation of (semi)-centralized networks, let us give a simple example. 
	Consider a four-node star network composed by $\{1,2,3,4\}$ where node $1$ is situated at the center and nodes $2$, $3$, and $4$ are (bidirectionally) connected with node $1$ but not connected to each other. In this case, the matrix $\mathbf{R}$ in our algorithm can be chosen as 
	{\small\[\mathbf{R}=\left[
		\begin{array}{cccc}
		1 & 0 &0&0\cr
		0.5& 0.5 & 0 & 0\cr
		0.5&0&0.5&0\cr
		0.5&0&0&0.5
		\end{array}\right], \ \
		\mathbf{C}=\left[
		\begin{array}{cccc}
		1&0.5&0.5&0.5\cr
		0&0.5&0&0\cr
		0&0&0.5&0\cr
		0&0&0&0.5
		\end{array}\right].
		\]}
	For a graphical illustration, the corresponding network topologies of $\Gra_\mathbf{R}$ and $\Gra_\mathbf{C}$ are shown in Fig. \ref{fig:Toy}.
	\begin{figure}[h]
		\begin{center}
			\includegraphics[height=8em,clip = true, trim = 0in 0in 0in 0in]{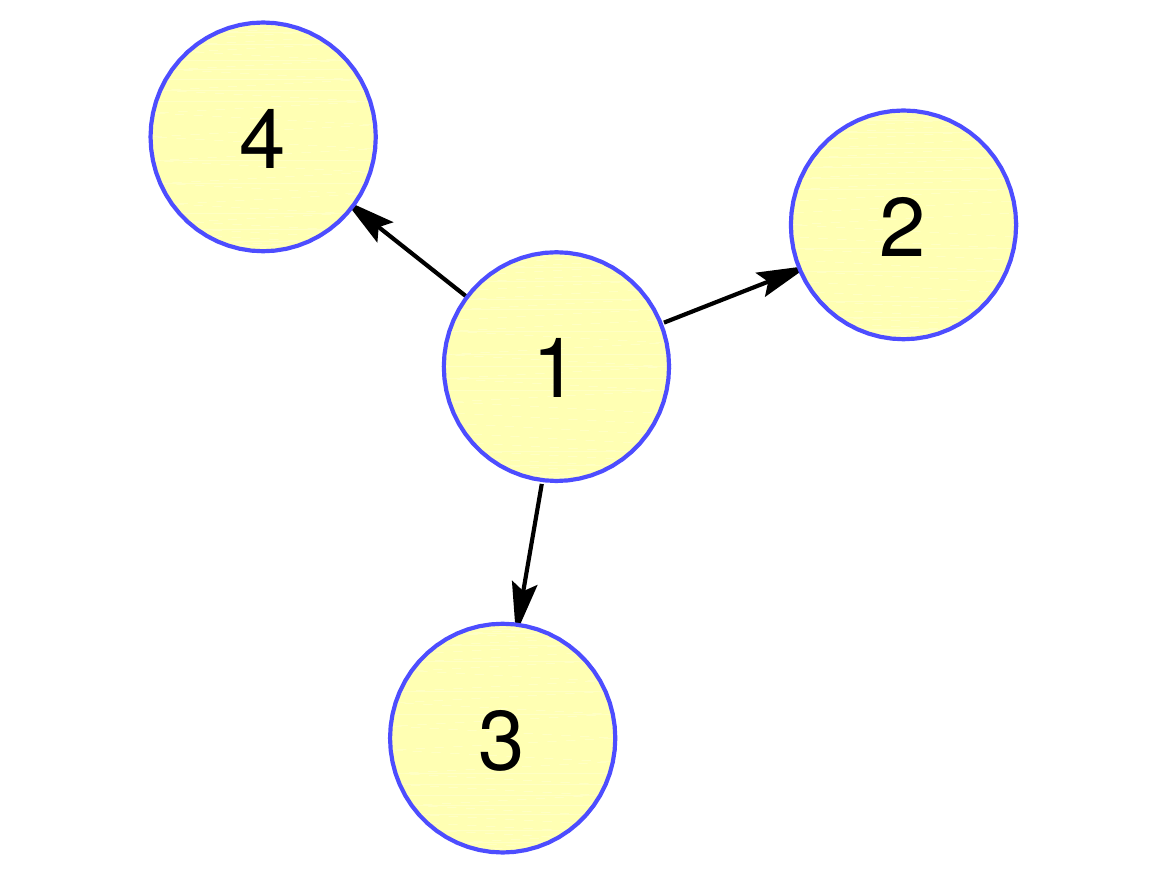}
			\includegraphics[height=8em,clip = true, trim = 0in 0in 0in 0in]{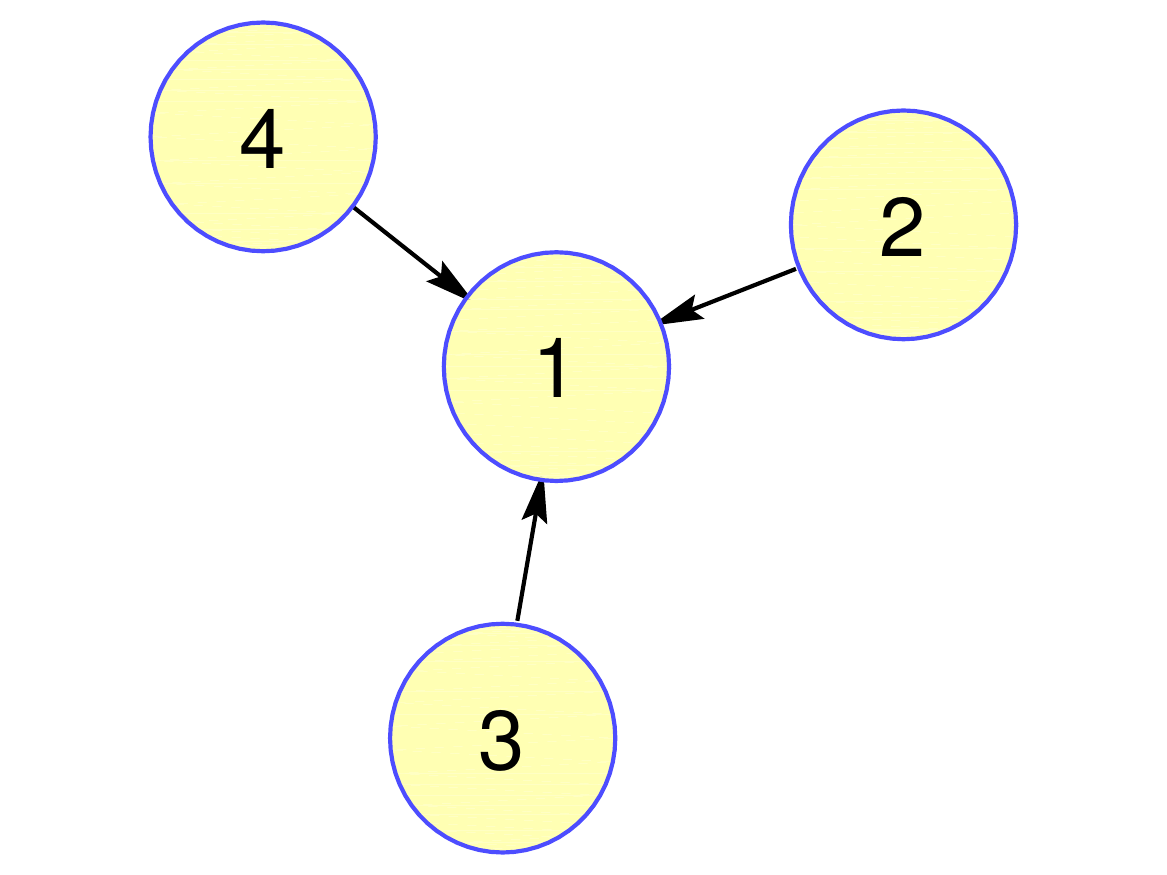}
			\caption{On the left is the graph $\Gra_\mathbf{R}$ and on the right is the graph $\Gra_\mathbf{C}$.}\label{fig:Toy}
		\end{center}
	\end{figure}
	The central node $1$'s information regarding $x_{1,k}$ is pulled by the neighbors (the entire network in this case) through $\Gra_\mathbf{R}$; the others only passively infuse the information from node $1$.
	At the same time, node $1$ has been pushed information regarding $y_{i,k}$ ($i=2,3,4$) from the neighbors
	through $\Gra_\mathbf{C}$; the other nodes only actively comply with the request from node $1$.
	This motivates the algorithm's name {\it push-pull gradient method}.
	Although nodes $2$, $3$, and $4$ are updating their $y_i$'s accordingly, these quantities do not have to contribute to the optimization procedure and will die out geometrically fast due to the weights in the last three rows of $\mathbf{C}$. Consequently, in this special case, the local stepsize $\alpha$ for agents $2$, $3$, and $4$ can be set to $0$. Without loss of generality, suppose $f_1(x)=0, \forall x$. Then the algorithm becomes a typical centralized algorithm for minimizing $\sum_{i=2}^4 f_i(x)$ where the master node $1$ utilizes the slave nodes $2$, $3$, and $4$ to compute the gradient information in a distributed way.
	
	
	Taking the above as an example for explaining the semi-centralized case, it is worth noting that node $1$ can be replaced by a strongly connected subnet in $\Gra_\mathbf{R}$ and $\Gra_\mathbf{C}$, respectively. 
	Correspondingly, nodes $2$, $3$, and $4$ can all be replaced by subnets as long as the information from the master layer in these subnets can be diffused to all the slave layer agents in $\Gra_\mathbf{R}$, while the information from all the slave layer agents can be diffused to the master layer in $\Gra_\mathbf{C}$.
	Specific requirements on connectivities of slave subnets can be understood by using the concept of rooted trees. We refer to the nodes as leaders if their roles in the network are similar to the role of node $1$; and the other nodes are termed as followers. Note that after the replacement of the individual nodes by subnets, the network structure in all subnets are decentralized, while the relationship between leader subnet and follower subnets is master-slave. This is why we refer to such an architecture as semi-centralized.
	\begin{remark}[A class of Push-Pull algorithms]\label{remark:ATC_variants}
		There can be multiple variants of the proposed algorithm depending on whether the Adapt-then-Combine (ATC) strategy \cite{Sayed2013} is used in the $\bx$-update and/or the $\by$-update (see Remark 3 in~\cite{nedic2017achieving} for more details). For readability, we only illustrate one algorithm in Algorithm 1 and call it Push-Pull in the above. We also generally use ``Push-Pull'' to refer to a class of algorithms regardless whether the ATC structure is used, if not causing confusion. Our forthcoming analysis can be adapted to these variants. Our numerical tests in Section \ref{sec: simulation} only involve some variants.
	\end{remark}
	
	\section{Convergence Analysis for Push-Pull}
	\label{sec: conv_analysis}
	In this section, we study the convergence properties of the proposed algorithm.
	We first define the following variables:
	\begin{eqnarray*}
		\ox_k  :=  \frac{1}{n}u^{\T} \mx_k,\qquad
		\oy_k  :=  \frac{1}{n}\mathbf{1}^{\T}\my_k.
	\end{eqnarray*}
	Our strategy is to bound $\|\ox_{k+1}-x^*\|_2$, $\|\mx_{k+1}-\mathbf{1}\ox_{k+1}\|_{\mrR}$ and $\|\my_{k+1}-v\oy_{k+1}\|_{\mrC}$ in terms of linear combinations of their previous values, where $\|\cdot\|_{\mrR}$ and $\|\cdot\|_{\mrC}$ are specific norms to be defined later. In this way we establish a linear system of inequalities which allows us to derive the convergence results. The proof technique was inspired by \cite{qu2017harnessing,xi2018linear}.
	
	Before diving into the detailed analysis, we present the main convergence result for the Push-Pull algorithm in~\eqref{algorithm P-P} in the following theorem.
	\begin{theorem}\label{theory:main}
		Suppose Assumptions \ref{asp; strconvex Lipschitz}-\ref{asp: nonempty root set} hold, $\alpha'\ge M\hat{\alpha}$ for some $M>0$ and
		\begin{equation}
		\label{bound_hat_alpha}
		\hat{\alpha} \le \min\left\{\frac{2c_3}{c_2+\sqrt{c_2^2+4c_1c_3}}, \frac{(1-\sigma_{\mrC})}{2\sigma_{\mrC}\delta_{\mrC,2}\|\mR\|_2 L}\right\},
		\end{equation}
		where  $c_1,c_2,c_3$ are given in (\ref{c1})-(\ref{c3}), and $\sigma_{\mrC}$ and $\delta_{\mrC,2}$ are defined in Lemma \ref{lem: special_norms} and Lemma \ref{lem: norm equivalence}, respectively.
		Then, the quantities $\|\ox_k-x^*\|_2$, $\|\mx_k-\mathbf{1}\ox_k\|_{\mrR}$ and $\|\my_k-v\oy_k\|_{\mrC}$ all converge to $0$ at the linear rate $\mathcal{O}(\rho(\mA)^k)$ 
		with $\rho(\mA)<1$, where $\rho(\mA)$ denotes the spectral radius of the matrix $\mA$ defined in (\ref{matrix_A}).
	\end{theorem}
	\begin{remark}
		Note that 
		\[\alpha'= \frac{1}{n}u^{\T}\balpha v=\sum_{i\in\mathcal{R}_\mathbf{R}\cap\mathcal{R}_{\mathbf{C}^{\T}}}\frac{1}{n}u_iv_i\alpha_i.\] 
		The condition $\alpha'\ge M\hat{\alpha}$ is automatically satisfied for a fixed $M$ in various situations. For example, if $\max_{i\in\mathcal{N}}\alpha_i=\max_{i\in\mathcal{R}_\mathbf{R}\cap\mathcal{R}_{\mathbf{C}^{\T}}}\alpha_i$ (which is always true when both $\mathcal{G}_{\mR}$ and $\mathcal{G}_{\mC^{\T}}$ are strongly connected), we can take $M=u_jv_j/n$ for $j=\arg\max_{i\in\mathcal{R}_\mathbf{R}\cap\mathcal{R}_{\mathbf{C}^{\T}}}\alpha_i$. For another example, if all $\alpha_i$ are equal, then $M=u^{\T}v/n$. 
		
		In general, the constant $M$ roughly measures the ratio of stepsizes used by agents in $\mathcal{R}_\mathbf{R}\cap\mathcal{R}_{\mathbf{C}^{\T}}$ and by all the agents. According to condition (\ref{bound_hat_alpha}) and definitions (\ref{c1})-(\ref{c3}), smaller $M$ leads to a tighter upper bound on the maximum stepsize $\hat{\alpha}$.
	\end{remark}
	\begin{remark}
		The upperbound on the stepsizes can be exactly calculated by (\ref{bound_hat_alpha}) assuming the weight matrices $\mathbf{R}$ and $\mathbf{C}$ are known. Note that since our proof technique is similar to those using the small-gain theorem (see e.g., \cite{qu2017harnessing,nedic2017achieving}), the upperbound obtained here may be conservative. Thus, it is still open how to develop new analytical tools to derive a tighter upperbound. However, it is worth mentioning that, in our numerical simulations, Push-Pull always allows for a very large region of stepsize choice compared to other distributed optimization algorithms applicable to directed graphs.
	\end{remark}
	\begin{remark}
		When $\hat{\alpha}$ is sufficiently small, it can be shown that $\rho(\mA)\simeq 1-\alpha'\mu$, in which case the Push-Pull algorithm is comparable to the centralized gradient descent method with stepsize $\alpha'$.\footnote{The proof of this statement is similar to that of Corollary 1 in \cite{pu2018distributed} and was omitted for conciseness of the paper.} Since $\alpha'=\sum_{i\in\mathcal{R}_\mathbf{R}\cap\mathcal{R}_{\mathbf{C}^{\T}}}\frac{1}{n}u_iv_i\alpha_i$, only the stepsizes of agents $i\in\mathcal{R}_\mathbf{R}\cap\mathcal{R}_{\mathbf{C}^{\T}}$ contribute to the convergence speed of Push-Pull. This corresponds to our discussion in Section \ref{sec: Push-Pull_unify}.
	\end{remark}
	
	\subsection{Preliminary Analysis}
	
	From the algorithm (\ref{algorithm P-P}) and Lemma \ref{lem: eigenvectors u v}, we have
	\begin{align}
	\label{ox_k+1 pre}
	\ox_{k+1}=\frac{1}{n}u^{\T} \mathbf{R}(\mx_k-\balpha \my_k)=\ox_k-\frac{1}{n}u^{\T}\balpha\my_k,
	\end{align}
	and
	\begin{multline}
	\label{oy_ave}
	\oy_{k+1}=  \frac{1}{n}\mathbf{1}^{\T} \left(\mathbf{C}\my_k+\nabla F(\mx_{k+1})-\nabla F(\mx_k)\right)\\
	= \oy_k+\frac{1}{n}\mathbf{1}^{\T}\left(\nabla F(\mx_{k+1})-\nabla F(\mx_k)\right).
	\end{multline}
	Let us further define $g_k:=\frac{1}{n}\mathbf{1}^{\T}\nabla F(\mathbf{1}\ox_k)$. Then, we obtain from relation (\ref{ox_k+1 pre}) that
	\begin{align}
	\label{ox pre}
	\ox_{k+1} & =\ox_k-\frac{1}{n}u^{\T}\balpha\left(\my_k-v\oy_k+v\oy_k\right)\notag\\
	& = \ox_k-\frac{1}{n}u^{\T}\balpha v\oy_k-\frac{1}{n}u^{\T}\balpha\left(\my_k-v\oy_k\right)\notag\\
	& = \ox_k-\alpha'g_k-\alpha'(\oy_k-g_k)-\frac{1}{n}u^{\T}\balpha\left(\my_k-v\oy_k\right),
	\end{align}
	where 
	\begin{equation}
	\label{alpha'}
	\alpha':=\frac{1}{n} u^{\T}\balpha v.
	\end{equation}
	We will show later that Assumptions \ref{asp: nonempty root set} and \ref{asp: step sizes} ensures $\alpha'>0$.
	
	In view of \eqref{algorithm P-P} and Lemma \ref{lem: eigenvectors u v}, using \eqref{ox_k+1 pre} we have
	\begin{multline}
	\label{mx-ox pre}
	\mx_{k+1}-\mathbf{1}\ox_{k+1}=\mathbf{R}(\mx_k-\balpha \my_k)-\mathbf{1}\ox_k+\frac{1}{n}\mathbf{1}u^{\T}\balpha\my_k\\
	=\mathbf{R}(\mx_k-\mathbf{1}\ox_k)-\left(\mathbf{R}-\frac{\mathbf{1}u^{\T}}{n}\right)\balpha\my_k\\
	=\left(\mathbf{R}-\frac{\mathbf{1}u^{\T}}{n}\right)(\mx_k-\mathbf{1}\ox_k)-\left(\mathbf{R}-\frac{\mathbf{1}u^{\T}}{n}\right)\balpha\my_k,
	\end{multline}
	and from \eqref{oy_ave} we obtain
	\begin{multline}
	\label{my-oy pre}
	\my_{k+1}-v\oy_{k+1}=\mathbf{C}\my_k-v\oy_k\\
	+\left(\mI-\frac{v\mathbf{1}^{\T}}{n}\right)\left(\nabla F(\mx_{k+1})-\nabla F(\mx_k)\right)\\
	=\left(\mathbf{C}-\frac{v\mathbf{1}^{\T}}{n}\right)(\my_k-v\oy_k)\\
	+\left(\mI-\frac{v\mathbf{1}^{\T}}{n}\right)\left(\nabla F(\mx_{k+1})-\nabla F(\mx_k)\right).
	\end{multline}
	
	\subsection{Supporting Lemmas}
	Before proceeding to prove the main result in Theorem \ref{theory:main}, we state a few useful lemmas.
	\begin{lemma}
		\label{lem: Lipschitz implications}
		Under Assumption \ref{asp; strconvex Lipschitz}, there holds
		\begin{equation*}
		\|\oy_k-g_k\|_2 \le \frac{L}{\sqrt{n}}\|\mx_k-\mathbf{1}\ox_k\|_2,\qquad
		\|g_k\|_2 \le L\|\ox_k-x^*\|_2.
		\end{equation*}
		In addition, when $\alpha'\le 2/(\mu+L)$, we have
		\begin{equation*}
		\|\ox_k-\alpha'g_k-x^*\|_2 \le (1-\alpha'\mu)\|\ox_k-x^*\|_2, \qquad \forall k.
		\end{equation*}
	\end{lemma}
	\begin{proof}
		See Appendix \ref{subsec: proof lemma Lipschitz implications}.
	\end{proof}
	\begin{lemma}
		\label{lem: spectral radii}
		Suppose Assumptions \ref{asp: stochastic}-\ref{asp: nonempty root set} hold. Let
		$\rho_{\mR}$ and $\rho_{\mC}$ be the spectral radii of $(\mathbf{R}-\mathbf{1}u^{\T}/n)$ and $(\mathbf{C}-v\mathbf{1}^{\T}/n)$, respectively. Then, we have $\rho_{\mR}<1$ and $\rho_{\mC}<1$.
	\end{lemma}
	\begin{proof}
		See Appendix \ref{subsec: proof lemma spectral radii}.
	\end{proof}
	\begin{lemma}
		\label{lem: special_norms}
		There exist matrix norms $\|\cdot\|_{\mrR}$ and $\|\cdot\|_{\mrC}$ 
		such that $\sigma_{\mrR}:=\|\mathbf{R}-\frac{\mathbf{1}u^{\T}}{n}\|_{\mrR}<1$, $\sigma_{\mrC}:=\|\mathbf{C}-\frac{v\mathbf{1}^{\T}}{n}\|_{\mrC}<1$, and $\sigma_{\mrR}$ and $\sigma_{\mrC}$ are arbitrarily close to $\rho_{\mR}$ and $\rho_{\mC}$, respectively. In addition, given any diagonal matrix $\mW\in\mathbb{R}^{n\times n}$, we have $\|\mW\|_{\mrR}=\|\mW\|_{\mrC}=\|\mW\|_2$.
	\end{lemma}
	\begin{proof}
		See \cite[Lemma 5.6.10]{horn1990matrix}  and the discussions thereafter.
	\end{proof}
	In the rest of this paper, with a slight abuse of notation,
	we do not distinguish between the vector norms on $\mathbb{R}^n$ and their induced
	matrix norms.
	\begin{lemma}
		\label{lem: matrix norm production}
		Given an arbitrary norm $\|\cdot\|$, for any $\mW\in\mathbb{R}^{n\times n}$ and $\mx\in\mathbb{R}^{n\times p}$, we have $\|\mW\mx\|\le \|\mW\|\|\mx\|$. For any $w\in\mathbb{R}^{n\times 1}$ and $x\in\mathbb{R}^{1\times p}$, we have $\|wx\|=\|w\|\|x\|_2$.
	\end{lemma}
	\begin{proof}
		See Appendix \ref{subsec: proof lemma matrix norm production}.
	\end{proof}
	\begin{lemma}
		\label{lem: norm equivalence}
		There exist constants $\delta_{\mrC,\mrR}, \delta_{\mrC,2}, \delta_{\mrR,\mrC}, \delta_{\mrR,2}>0$ such that for all $\mx\in\mathbb{R}^{n\times p}$, we have  $\|\mx\|_{\mrC}\le \delta_{\mrC,\mrR}\|\mx\|_{\mrR}$, $\|\mx\|_{\mrC}\le \delta_{\mrC,2}\|\mx\|_2$, $\|\mx\|_{\mrR}\le \delta_{\mrR,\mrC}\|\mx\|_{\mrC}$,  and $\|\mx\|_{\mrR}\le \delta_{\mrR,2}\|\mx\|_2$. In addition, with 
		a proper rescaling of the norms $\|\cdot\|_{\mrR}$ and $\|\cdot\|_{\mrC}$, we have
		$\|\mx\|_2\le \|\mx\|_{\mrR}$ and $\|\mx\|_2\le \|\mx\|_{\mrC}$ for all $\mx$.
	\end{lemma}
	\begin{proof}
		The above result follows from the equivalence relation of all norms on $\mathbb{R}^n$ and Definition \ref{def: norm n p}.
	\end{proof}
	
	The following critical lemma establishes a linear system of inequalities that bound $\|\ox_{k+1}-x^*\|_2$, $\|\mx_{k+1}-\mathbf{1}\ox_k\|_{\mrR}$ and $\|\my_{k+1}-v\oy_k\|_{\mrC}$.
	\begin{lemma}
		\label{lem: important inequalities}
		Under Assumptions \ref{asp; strconvex Lipschitz}-\ref{asp: nonempty root set}, when $\alpha'\le 2/(\mu+L)$, we have the following linear system of inequalities:
		\begin{equation}
		\label{main ineqalities}
		\begin{array}{l}
		\begin{bmatrix}
		\|\ox_{k+1}-x^*\|_2\\
		\|\mx_{k+1}-\mathbf{1}\ox_{k+1}\|_{\mrR}\\
		\|\my_{k+1}-v\oy_{k+1}\|_{\mrC}
		\end{bmatrix}
		\le
		\mA
		\begin{bmatrix}
		\|\ox_k-x^*\|_2\\
		\|\mx_k-\mathbf{1}\ox_k\|_{\mrR}\\
		\|\my_k-v\oy_k\|_{\mrC}
		\end{bmatrix},
		\end{array}
		\end{equation}
		where the inequality is taken component-wise, and elements of the transition matrix $\mathbf{A}=[a_{ij}]$ are given by:
		\begin{align}
		\label{matrix_A}
		\hspace{-0.6em}\begin{array}{l}
		\begin{bmatrix}
		a_{11}\\
		a_{21}\\
		a_{31}
		\end{bmatrix} = 
		\begin{bmatrix}
		1-\alpha'\mu\\
		\hat{\alpha}\sigma_{\mrR}\|v\|_{\mrR} L\\
		\hat{\alpha}c_0 \delta_{\mrC,2}\|\mR\|_2\|v\|_2 L^2
		\end{bmatrix},\\
		\begin{bmatrix}
		a_{12}\\
		a_{22}\\
		a_{32}
		\end{bmatrix} =
		\begin{bmatrix}
		\frac{\alpha'L}{\sqrt{n}}\\
		\sigma_{\mrR}\left(1+\hat{\alpha}\|v\|_{\mrR} \frac{L}{\sqrt{n}}\right)\\
		c_0 \delta_{\mrC,2}L\left(\|\mR-\mI\|_2+\hat{\alpha}\|\mR\|_2\|v\|_2\frac{L}{\sqrt{n}}\right)
		\end{bmatrix},\\
		\begin{bmatrix}
		a_{13}\\
		a_{23}\\
		a_{33}
		\end{bmatrix} =
		\begin{bmatrix}
		\frac{\hat{\alpha}\|u\|_2}{n}\\
		\hat{\alpha}\sigma_{\mrR}\delta_{\mrR,\mrC}\\
		\sigma_{\mrC}+ \hat{\alpha}c_0\delta_{\mrC,2}\|\mR\|_2 L
		\end{bmatrix},
		\end{array}
		\end{align}
		where $\hat{\alpha}:=\max_i\alpha_i$ and $c_0:=\|\mI-v\mathbf{1}^{\T}/n\|_{\mrC}$.
	\end{lemma}
	\begin{proof}
		See Appendix \ref{proof: important inequalities}.
	\end{proof}
	
	In light of Lemma \ref{lem: important inequalities}, $\|\ox_k-x^*\|_2$, $\|\mx_k-\mathbf{1}\ox_k\|_{\mrR}$ and $\|\my_k-v\oy_k\|_{\mrC}$ all converge to $0$ linearly at rate $\mathcal{O}(\rho(\mA)^k)$ if the spectral radius of $\mA$ satisfies $\rho(\mA)<1$. 
	The next lemma provides some sufficient conditions for the relation $\rho(\mA)<1$ to hold.
	\begin{lemma}
		\label{lem: rho_M}
		\cite[Lemma 5]{pu2018distributed} Given a nonnegative, irreducible matrix $\mM=[m_{ij}]\in\mathbb{R}^{3\times 3}$ with $m_{ii}<\lambda^*$ for some $\lambda^*>0$ for all $i=1,2,3$. A necessary and sufficient condition for $\rho(\mM)<\lambda^*$ is $\mathrm{det}(\lambda^* \mathbf{I}-\mM)>0$.
	\end{lemma}
	
	\subsection{Proof of Theorem \ref{theory:main}}
	In light of Lemma \ref{lem: rho_M}, it suffices to ensure $a_{11},a_{22},a_{33}<1$ and $\mathrm{det}(\mI-\mA)>0$, or
	\begin{align}
	\label{|I-A|>0}
	\begin{array}{l}
	\quad\mathrm{det}(\mI-\mA)=(1-a_{11})(1-a_{22})(1-a_{33})-a_{12}a_{23}a_{31}\\
	\quad-a_{13}a_{21}a_{32}-(1-a_{22})a_{13}a_{31}-(1-a_{11})a_{23}a_{32}\\
	\quad-(1-a_{33})a_{12}a_{21}\\
	= (1-a_{11})(1-a_{22})(1-a_{33})\\
	\quad-\alpha'\hat{\alpha}^2\sigma_{\mrR}c_0\delta_{\mrR,\mrC}\delta_{\mrC,2}\|\mR\|_2 \|v\|_2\frac{L^3}{\sqrt{n}}\\
	\quad-\hat{\alpha}^2\sigma_{\mrR}c_0\delta_{\mrC,2} \|u\|_2\|v\|_{\mrR}\left(\|\mR-\mI\|_2+\hat{\alpha}\|\mR\|_2 \|v\|_2\frac{L}{\sqrt{n}}\right)\frac{L^2}{n}\\
	\quad-\hat{\alpha}^2 c_0\delta_{\mrC,2}\|\mR\|_2 \|v\|_2\|u\|_2\frac{L^2}{n}(1-a_{22})\\
	\quad-\hat{\alpha}\sigma_{\mrR}c_0 \delta_{\mrR,\mrC}\delta_{\mrC,2}L\left(\|\mR-\mI\|_2+\hat{\alpha}\|\mR\|_2 \|v\|_2\frac{L}{\sqrt{n}}\right)(1-a_{11})\\
	\quad-\alpha'\hat{\alpha}\sigma_{\mrR} \|v\|_{\mrR}\frac{L^2}{\sqrt{n}}(1-a_{33})>0.
	\end{array}
	\end{align}
	We now provide some sufficient conditions under which $a_{11},a_{22},a_{33}<1$ and (\ref{|I-A|>0}) holds true.
	First, $a_{11}<1$ is ensured by choosing $\alpha'\le 2/(\mu+L)$. Suppose
	$1-a_{22} \ge (1-\sigma_{\mrR})/2$ and
	$1-a_{33} \ge (1-\sigma_{\mrC})/2$ under properly chosen stepsizes. These relations holds when
	\begin{align}
	\label{alpha loose condition}
	\hat{\alpha}\le \min\left\{\frac{(1-\sigma_{\mrR})\sqrt{n}}{2\sigma_{\mrR}\|v\|_{\mrR} L},\frac{(1-\sigma_{\mrC})}{2c_0\delta_{\mrC,2}\|\mR\|_2 L}\right\}.
	\end{align}
	
	Second, a sufficient condition for $\mathrm{det}(\mI-\mA)>0$ is to substitute $(1-a_{22})$ (resp., $(1-a_{33})$) in (\ref{|I-A|>0}) by $(1-\sigma_{\mrR})/2$ (resp., $(1-\sigma_{\mrC})/2$) and take $\alpha'=M\hat{\alpha}$. We then have $c_1\hat{\alpha}^2+c_2\hat{\alpha}-c_3<0$,
	where
	\begin{align}
	\label{c1}
	\begin{array}{l}
	\quad c_1 = M\sigma_{\mrR}c_0\delta_{\mrR,\mrC}\delta_{\mrC,2}\|\mR\|_2 \|v\|_2\frac{L^3}{\sqrt{n}}\\
	\quad+\sigma_{\mrR}c_0\delta_{\mrC,2} \|u\|_2\|v\|_{\mrR}\|\mR\|_2 \|v\|_2\frac{L^3}{n\sqrt{n}}\\
	\quad+M\mu\sigma_{\mrR}c_0 \delta_{\mrR,\mrC}\delta_{\mrC,2}\|\mR\|_2 \|v\|_2\frac{L^2}{\sqrt{n}}\\
	=\sigma_{\mrR}c_0\delta_{\mrC,2}\|\mR\|_2 \|v\|_2\frac{L^2}{n\sqrt{n}}\left[M\delta_{\mrR,\mrC}n(L+\mu)+\|u\|_2\|v\|_{\mrR} L\right],
	\end{array}
	\end{align}
	\begin{align}
	\label{c2}
	\begin{array}{l}
	c_2=\sigma_{\mrR}c_0\delta_{\mrC,2} \|u\|_2\|v\|_{\mrR}\|\mR-\mI\|_2\frac{L^2}{n}\\
	\qquad+c_0\delta_{\mrC,2}\|\mR\|_2 \|v\|_2\|u\|_2(1-\sigma_{\mrR})\frac{L^2}{2n}\\
	\qquad+M\sigma_{\mrR}c_0\delta_{\mrR,\mrC}\delta_{\mrC,2}\|\mR-\mI\|_2 \mu L\\
	\qquad+\frac{M}{2}\sigma_{\mrR} \|v\|_{\mrR}(1-\sigma_{\mrC})\frac{L^2}{\sqrt{n}},
	\end{array}
	\end{align}
	and 
	\begin{align}
	\label{c3}
	\begin{array}{l}
	c_3=\frac{M}{4}(1-\sigma_{\mrC})(1-\sigma_{\mrR})\mu.
	\end{array}
	\end{align}
	Hence
	\begin{align}
	\label{alpha strict condition}
	\hat{\alpha}\le \frac{2c_3}{c_2+\sqrt{c_2^2+4c_1c_3}}.
	\end{align}
	Relations (\ref{alpha loose condition}) and (\ref{alpha strict condition}), together with the fact that $\mathbf{A}$ is irreducible (a $3\times 3$ full matrix), yield the final bound on $\hat{\alpha}$.

	\section{A Gossip-Like Push-Pull Method (G-Push-Pull)}
	\label{sec: G-Push-Pull}
	
	In this section, we introduce a generalized random-gossip push-pull algorithm. We call it G-Push-Pull and outline it in the following (Algorithm 2)\footnote{In the algorithm description, the multiplication sign ``$\times$'' is added simply for avoiding visual confusion. It still represents the commonly recognized scalar-scalar or scalar-vector multiplication.}.
	
	\smallskip	
	\begin{table}
		\normalsize
		\centering
		{\textbf{Algorithm 2: G-Push-Pull}}
		
		\smallskip
		\begin{tabularx}{0.5\textwidth}{X}
			\hline
			Each agent $i$ chooses its local step size $\alpha_i\ge 0$;\\
			Each agent $i$ initializes with any arbitrary $x_{i,0}\in\R^p$ and $y_{i,0}=\nabla f_i(x_{i,0})$;\\
			\textbf{for} time slot $k=0,1,\cdots$ \textbf{do}\\
			\quad agent $i_k$ is uniformly randomly selected from $\mathcal{N}$;\\
			\quad agent $i_k$ uniformly randomly chooses the set $\outneighbor{\mathbf{R},i_k,k}$, a subset of its out-neighbors in $\Gra_\mathbf{R}$ at ``time'' $k$;\\
			\quad agent $i_k$ sends $x_{i_k,k}$ to all members in $\outneighbor{\mathbf{R},i_k,k}$;\\
			\quad every agent $j_k$ from $\outneighbor{\mathbf{R},i_k,k}$ generates $\gamma_{\mathbf{R},j_k,k}\in(0,1)$;\\
			\quad agent $i_k$ uniformly randomly chooses the set $\outneighbor{\mathbf{C},i_k,k}$, a subset of its out-neighbors in $\Gra_\mathbf{C}$ at ``time'' $k$;\\
			\quad agent $i_k$ sends $\gamma_{\mathbf{C},i_k,k}\times y_{i_k,k}$ to all members in $\outneighbor{\mathbf{C},i_k,k}$, where $\gamma_{\mathbf{C},i_k,k}$ is generated at agent $i_k$ such that $\gamma_{\mathbf{C},i_k,k}|\outneighbor{\mathbf{C},i_k,k}|<1$;\\
			\qquad$x_{i_k,k+1} = x_{i_k,k}-\alpha_{i_k} y_{i_k,k}$;\\
			\qquad \textbf{for} all $j_k\in\outneighbor{\mathbf{R},i_k,k}$ and $l_k\in\outneighbor{\mathbf{C},i_k,k}$ \textbf{do}\\
			\qquad \qquad \textbf{if} $j_k==l_k$\\
			\qquad $x_{j_k,k+1} = (1-\gamma_{\mathbf{R},j_k,k})x_{j_k,k}+\gamma_{\mathbf{R},j_k,k}~x_{i_k,k}-2\alpha_{j_k} y_{j_k,k}$;\\
			\qquad \qquad \textbf{else}\\
			\qquad  $x_{j_k,k+1} = (1-\gamma_{\mathbf{R},j_k,k})x_{j_k,k}+\gamma_{\mathbf{R},j_k,k}~x_{i_k,k}-\alpha_{j_k} y_{j_k,k}$;\\
			\qquad  $x_{l_k,k+1} = x_{l_k,k}-\alpha_{l_k} y_{l_k,k}$;\\	
			\qquad \qquad \textbf{end if}\\
			\qquad \textbf{end for}\\
			\qquad $y_{i_k,k+1} =  (1-\gamma_{\mathbf{C},i_k,k}|\outneighbor{\mathbf{C},i_k,k}|)y_{i_k,k}+\nabla f_{i_k}(x_{i_k,k+1})-\hfill~~~~~~~~~~~~~~~~~~~~\nabla f_{i_k}(x_{i_k,k})$;\\		
			\qquad \textbf{for} all $l_k\in\outneighbor{\mathbf{C},i_k,k}$ \textbf{do}\\
			\qquad $y_{l_k,k+1} = y_{l_k,k}+\gamma_{\mathbf{C},i_k,k}\times y_{i_k,k}+\nabla f_{l_k}(x_{l_k,k+1})-\hfill~~~~~~~~~~~~~~~~~~~~\nabla f_{l_k}(x_{l_k,k})$;\\
			\qquad \textbf{end for}\\
			\qquad \textbf{for} all $j_k\in\outneighbor{\mathbf{R},i_k,k}$ but $j_k\notin\outneighbor{\mathbf{C},i_k,k}$ \textbf{do}\\
			\qquad \qquad $y_{j_k,k+1} = y_{j_k,k}+\nabla f_{j_k}(x_{j_k,k+1})-\nabla f_{j_k}(x_{j_k,k})$;\\
			\qquad \textbf{end for}\\		
			\textbf{end for}\\
			\hline
		\end{tabularx}
	\end{table}
	\smallskip	
	
	Algorithm 2 illustrates the G-Push-Pull algorithm. At each ``time slot'' $k$, it is possible in practice that multiple agents (entities that are equivalent to the agent ``$i_k$'' employed in the algorithm) are activated/selected. This random selection process is done by placing a Poisson clock on each agent. Anytime when a node is awakened by itself or push-notified (or pull-alerted), it will be temporarily locked for the current paired update. We note that in this gossip version (Algorithm 2), only the push-communication-protocol is employed. Other possible variants that involve only the pull-communication-protocol or both protocols exist. For instance, to give a visual impression, for a 4-agent network connected as $1\rightarrow2\rightarrow3\rightarrow4\rightarrow 1$ and $1\rightarrow3$ (each arrow represents a unidirectional data link and this digraph is not balanced/regular)\footnote{For simplicity, we assume $\mathcal{G}_{\mrR}=\mathcal{G}_{\mrC}$ in this example.}, if we are to design a pull-only gossip algorithm and suppose agent $3$ is updating (pulling information from $1$ and $2$) at time $k$, the mixing matrices can be designed/implemented as
	{\small\[\mathbf{R}_k=\left[
		\begin{array}{cccc}
		1 & 0 &0&0\cr
		0& 1 & 0 & 0\cr
		1/3&1/3&1/3&0\cr
		0&0&0&1
		\end{array}\right], \ \
		\mathbf{C}_k=\left[
		\begin{array}{cccc}
		1/2&0&0&0\cr
		0&1/2&0&0\cr
		1/2&1/2&1&0\cr
		0&0&0&1
		\end{array}\right].
		\]}\normalsize
	From the third rows of $\mathbf{R}_k$ and $\mathbf{C}_k$, we can see that agent $3$ is aggregating the pulled information ($x_{1,k}$, $x_{2,k}$, $1/2\times y_{1,k}$, and $1/2\times y_{2,k}$); from the first and second column of $\mathbf{C}_k$, we can observe that agents $1$ and $2$ are ``sharing part of $y$'' and rescaling their own $y$. The gossip mechanism allows and is in favor of a push-only or pull-only network, which is different from what we require for the general static network carrying Algorithm 1 (see Section \ref{sec: Push-Pull}, the discussion right before Section \ref{sec: Push-Pull_unify}). Such difference is due to the fact that in gossip algorithms, at each ``iteration'' $k$, only one or multiple isolated trees with depth $1$ are activated, thus trivial weights assignment mechanisms exist in the $\mathbf{C}_k$ graph. For instance in the above example with a 4-agent network, the chosen $\mathbf{C}_k$ could be generated by letting the agents being pulled simply ``halve the $y$ variable before using it or sending it out''. This trick for gossip algorithms is difficult, if not impossible, to implement in a synchronized network with other general topologies.
	
	In the following, to make the convergence analysis concise, we further assume/restrict to the situation where $\alpha_i=\alpha>0$ for all $i\in\mathcal{N}$, $|\outneighbor{\mathbf{C},i_k,k}|\leq1$, $|\outneighbor{\mathbf{R},i_k,k}|\leq1$, and $\gamma_{\mathbf{C},i_k,k}=\gamma_{\mathbf{R},i_k,k}=\gamma$ for all $i_k\in\mathcal{N}$ and all $k=0,1,\ldots$. 
	With the simplification, we can represent the recursion of G-Push-Pull in a compact matrix form:
	\begin{subequations}\label{algorithm G-P-P}
		\begin{align}
		&\mx_{k+1} =  \mR_k\mx_k-\alpha\mQ_k \my_k,\label{eq:x-update GPP}\\
		&\my_{k+1} =  \mC_k\my_k+\nabla F(\mx_{k+1})-\nabla F(\mx_k),\label{eq:y-update GPP}
		\end{align}
	\end{subequations}
	where the matrices $\mR_k$ and $\mC_k$ are given by
	\begin{subequations}\label{R_k C_k}
		\begin{align}
		& \mR_k =  \mI+\gamma\left(e_{j_k}e_{i_k}^{\T}-e_{j_k}e_{j_k}^{\T}\right),\\
		& \mC_k =  \mI+\gamma\left(e_{l_k}e_{i_k}^{\T}-e_{i_k}e_{i_k}^{\T}\right),
		\end{align}
	\end{subequations}
	respectively. Here, $e_i=[0, \cdots, 0, 1, 0, \cdots]^{\T}\in\mathbb{R}^{n\times 1}$ is a unit vector with the $i$th component equal to $1$. Notice that each $\mR_k$ is row-stochastic, and each $\mC_k$ is column-stochastic. The random matrix variable $\mQ_k=\dia{e_{i_k}+e_{j_k}+e_{l_k}}$.
	
	\begin{remark}
		In practice, after receiving information from agent $i_k$ at step $k$, agents $j_k$ and $l_k$ can choose to perform their updates when they wake up in a future step.
	\end{remark}
	
	
	\section{Convergence Analysis for G-Push-Pull}
	\label{sec: conv_analysis_G-Push-Pull}
	
	Define $\bar{\mR}:=\bE[\mR_k]$ and $\bar{\mC}:=\bE[\mC_k]$. Denote by $\bar{u}^{\T}$ ($\bar{u}^{\T}\mathbf{1}=n$) the left eigenvector of $\bar{\mR}$ w.r.t. eigenvalue $1$, and let $\bar{v}$ ($\mathbf{1}^{\T}\bar{v}=n$) be the right eigenvector of $\bar{\mC}$ w.r.t. eigenvalue $1$. Let
	$\bar{x}_k=\frac{1}{n}\bar{u}^{\T}\mx_k$ and $\bar{y}_k=\frac{1}{n}\mathbf{1}^{\T}\my_k$ as before.
	Our strategy is to bound $\bE[\|\ox_{k+1}-x^*\|_2^2]$, $\bE[\|\mx_{k+1}-\mathbf{1}\ox_k\|_{\mrS}^2]$ and $\bE[\|\my_{k+1}-\bar{v}\oy_k\|_{\mrD}^2]$ in terms of linear combinations of their previous values, where $\|\cdot\|_{\mrS}$ and $\|\cdot\|_{\mrD}$ are norms to be specified later. Then based on the established linear system of inequalities, we prove the convergence of G-Push-Pull.
	
	We first state the main convergence result for G-Push-Pull in the following theorem.
	\begin{theorem}\label{theory:main_gpp}
		Suppose Assumptions \ref{asp; strconvex Lipschitz}-\ref{asp: nonempty root set} hold and
		\begin{subequations}\label{condition_gpp}
			\begin{align}
			& \gamma< \min\left\{\bar{\gamma}_{\mrR}, \bar{\gamma}_{\mrC}, \frac{\eta\mu\sqrt{(1-\sigma_{\bar{\mrC}})}}{8L\sqrt{d_4 d_7}}\right\}, \label{gamma_condition_gpp}\\
			& \alpha \le \min\left\{\frac{2c_6}{c_5+\sqrt{c_5^2+4c_4c_6}}, \right\},
			\end{align}
		\end{subequations}
		where the constant $\bar{\gamma}_{\mrR}>0$ is defined in Lemma \ref{lem: x_consensus gpp}, $\bar{\gamma}_{\mrC},\sigma_{\bar{\mrC}}>0$ are defined in Lemma \ref{lem: y_alignment gpp}, $\eta>0$ is given in Lemma \ref{lem: x-xstar_pre gpp}, $c_4$-$c_6$ are given in (\ref{c4})-(\ref{c6}), and $d_4,d_7>0$ are defined in (\ref{d1d7}).
		Then, the quantities $\bE[\|\ox_k-x^*\|_2^2]$, $\bE[\|\mx_k-\mathbf{1}\ox_k\|_{\mrS}^2]$ and $\bE[\|\my_k-v\oy_k\|_{\mrD}^2]$ all converge to $0$ at the linear rate $\mathcal{O}(\rho(\mB)^k)$, where $\rho(\mB)<1$ denotes the spectral radius of the matrix $\mB$ defined in (\ref{matrix_B}).
	\end{theorem}
	
	\subsection{Preliminaries}
	
	From (\ref{algorithm G-P-P}), we have the following recursive relations.
	\begin{multline*}
	\bar{x}_{k+1}=\frac{\bar{u}^{\T}}{n}\left(\mR_k\mx_k-\alpha\mQ_k \my_k\right)\\
	=\bar{x}_k-\alpha\frac{\bar{u}^{\T}}{n}\mQ_k \my_k+\frac{\bar{u}^{\T}}{n}\left(\mR_k-\bar{\mR}\right)\left(\mx_k-\frac{\mathbf{1}\bar{u}^{\T}}{n}\mx_k\right),
	\end{multline*}
	\begin{equation}\label{mx ox pre gpp}
	\mx_{k+1}-\frac{\mathbf{1}\bar{u}^{\T}}{n}\mx_{k+1}=\left(\mI-\frac{\mathbf{1}\bar{u}^{\T}}{n}\right)\mR_k\mx_k-\left(\mI-\frac{\mathbf{1}\bar{u}^{\T}}{n}\right)\alpha\mQ_k\my_k,
	\end{equation}
	and
	\begin{multline}\label{my oy pre gpp}
	\my_{k+1}-\frac{\bar{v} \mathbf{1}^{\T}}{n}\my_{k+1}=\left(\mI-\frac{\bar{v} \mathbf{1}^{\T}}{n}\right)\mC_k\my_k\\
	+\left(\mI-\frac{\bar{v} \mathbf{1}^{\T}}{n}\right)\left[\nabla F(\mx_{k+1})-\nabla F(\mx_k)\right].
	\end{multline}
	
	To derive a linear system of inequalities from the above equations, we first provide some useful facts about $\mR_k$ and $\mC_k$ as well as their expectations $\bar{\mR}$ and $\bar{\mC}$.
	Let
	\begin{align*}
	\mT_k := e_{j_k}e_{i_k}^{\T}-e_{j_k}e_{j_k}^{\T},\qquad
	\mE_k := e_{l_k}e_{i_k}^{\T}-e_{i_k}e_{i_k}^{\T}.
	\end{align*}
	Then from (\ref{R_k C_k}) we have $\mR_k=\mI+\gamma\mT_k$ and $\mC_k=\mI+\gamma\mE_k$.
	Define $\bar{\mT}=\bE[\mT_k]$ and $\bar{\mE}=\bE[\mE_k]$. We obtain
	\begin{align*}
	\bar{\mR}=\mI+\gamma\bar{\mT},\qquad
	\bar{\mC}=\mI+\gamma\bar{\mE}.
	\end{align*}
	Matrices $\bar{\mT}$ and $\bar{\mE}$ have the following algebraic property.
	\begin{lemma}
		\label{lem: eigenvalues_bar_T E}
		The matrix $\bar{\mT}$ (resp., $\bar{\mE}$) has a unique eigenvalue $0$; all the other eigenvalues lie in the unit circle centered at $(-1,0)\in\mathbb{C}^2$.
	\end{lemma}
	\begin{proof}
		Note that $\mI+\bar{\mT}$ is a nonnegative row-stochastic matrix corresponding to the graph $\Gra_{\mathbf{R}}$. It has spectral radius $1$, which is also the unique eigenvalue of modulus $1$ due to the existence of a spanning tree in the graph $\Gra_{\mathbf{R}}$~\cite[Lemma 3.4]{ren2005consensus}. Therefore, $0$ is a unique eigenvalue of $\bar{\mT}$, and all the other eigenvalues lie in the unit circle centered at $(-1,0)\in\mathbb{C}^2$.
		The argument for $\bar{\mE}$ is similar.
	\end{proof}
	
	Note that $\bar{u}^{\T}$ is the left eigenvector of $\bar{\mT}$ w.r.t. eigenvalue $0$, and $\mathbf{1}^{\T}$ is the left eigenvector of $\bar{\mE}$ w.r.t. eigenvalue $0$. We have the following result.
	\begin{lemma}
		\label{lem: TE decomposition}
		The matrix $\bar{\mT}$ can be decomposed as $\bar{\mT}=\mS^{-1}\mJ_{\mrT}\mS$, where $\mJ_{\mrT}\in \mathbb{C}^{n\times n}$ has $0$ on its top-left, and it differs from the Jordan form of $\bar{\mT}$ only on the superdiagonal entries\footnote{If $\bar{\mT}$ is diagonalizable, then $\mJ_{\mrT}$ is exactly the Jordan form of $\bar{\mT}$. The same relation applies to $\mJ_{\mrE}$ and $\bar{\mE}$ in the next paragraph.}. Square matrix $\mS$ has $\bar{u}^{\T}$ as its first row.
		The rows of $\mS$ are either left eigenvectors of $\bar{\mT}$, or generalized left eigenvectors of $\bar{\mT}$ (up to rescaling). In particular, the superdiagonal elements of $\mJ_{\mrT}$ can be made arbitrarily close to $0$ by proper choice of $\mS$. 
		
		The matrix $\bar{\mE}$ can be decomposed as $\bar{\mE}=\mD^{-1}\mJ_{\mrE}\mD$, where $\mJ_{\mrE}\in \mathbb{C}^{n\times n}$ has $0$ on its top-left, and it differs from the Jordan form of $\bar{\mE}$ only on the superdiagonal entries. Square matrix $\mD$ has $\mathbf{1}^{\T}$ as its first row. 
		The rows of $\mD$ are either left eigenvectors of $\bar{\mE}$, or generalized left eigenvectors of $\bar{\mE}$ (up to rescaling). In particular, the superdiagonal elements of $\mJ_{\mrE}$ can be made arbitrarily close to $0$ by proper choice of $\mD$. 
	\end{lemma}
	\begin{proof}
		Since $\bar{u}^{\T}$ is the left eigenvector of $\bar{\mT}$ w.r.t. eigenvalue $0$, the Jordan form of $\bar{\mT}$ can be written as $\bar{\mT}=\tilde{\mS}^{-1}\tilde{\mJ}_{\mrT}\tilde{\mS}$, where $\tilde{\mJ}_{\mrT}$ has $0$ on its top-left, and $\bar{u}^{\T}$ is the first row in $\tilde{\mS}$ (see \cite{horn1990matrix}). If $\bar{\mT}$ is diagonalizable, then the matrix $\tilde{\mJ}_{\mrT}$ is diagonal, in which case we can take $\mJ_{\mrT}=\tilde{\mJ}_{\mrT}$ and $\mS=\tilde{\mS}$. If not, the matrix $\tilde{\mJ}_{\mrT}$ has superdiagonal elements equal to $1$. Then we let $\mS$ to be different from $\tilde{\mS}$ only in the rows corresponding to the generalized left eigenvectors of $\bar{\mT}$. By scaling down these rows, the superdiagonal elements of $\mJ_{\mrT}$ can be made arbitrarily close to $0$.
		The proof for $\bar{\mE}$ is similar.
	\end{proof}
	
	The following lemma is the final cornerstone we need to build our proof for the main results.
	\begin{lemma}
		\label{lem: Jordan form RC}
		Suppose $\bar{\mT}=\mS^{-1}\mJ_{\mrT}\mS$ and $\bar{\mE}=\mD^{-1}\mJ_{\mrE}\mD$ (as described in Lemma \ref{lem: TE decomposition}).
		Then the matrix $\bar{\mR}-\frac{\mathbf{1}\bar{u}^{\T}}{n}$ has decomposition
		\begin{equation*}
		\bar{\mR}-\frac{\mathbf{1}\bar{u}^{\T}}{n}=\mS^{-1}\mJ_{\mrR}\mS,
		\end{equation*}
		where $\mJ_{\mrR}=\mI-\dia{e_1}+\gamma\mJ_{\mrT}$,
		and the matrix $\bar{\mC}-\frac{\bar{v}\mathbf{1}^{\T}}{n}$ has decomposition
		\begin{equation*}
		\bar{\mC}-\frac{\bar{v}\mathbf{1}^{\T}}{n}=\mD^{-1}\mJ_{\mrC}\mD,
		\end{equation*}
		where $\mJ_{\mrC}=\mI-\dia{e_1}+\gamma\mJ_{\mrE}$. 
	\end{lemma}
	\begin{proof}
		Recall that the rows of $\mS$ are left (generalized) eigenvectors of $\bar{\mT}$ (up to rescaling).
		Since $\mathbf{1}$ is the right eigenvector of $\bar{\mT}$ w.r.t eigenvalue $0$, it is orthogonal to the left (generalized) eigenvectors of $\bar{\mT}$ w.r.t eigenvalues other than $0$.
		Thus we have $\mS\frac{\mathbf{1}\bar{u}^{\T}}{n}=e_1\bar{u}^{\T}=\dia{e_1}\mS$.
		Therefore,
		\begin{multline*}
		\bar{\mR}-\frac{\mathbf{1}\bar{u}^{\T}}{n}=\mS^{-1}(\mI+\gamma\mJ_{\mrT})\mS-\mS^{-1}\dia{e_1}\mS\\
		=\mS^{-1}(\mI-\dia{e_1}+\gamma\mJ_{\mrT})\mS.
		\end{multline*}
		Similarly, we can prove the second relation.
	\end{proof}
	
	In the rest of this section, we assume that $\bar{\mT}=\mS^{-1}\mJ_{\mrT}\mS$ and $\bar{\mE}=\mD^{-1}\mJ_{\mrE}\mD$ for some fixed matrices $\mS$, $\mD$, $\mJ_{\mrT}$ and $\mJ_{\mrE}$ as described in Lemma \ref{lem: TE decomposition}. In particular, $\mJ_{\mrT}$ and $\mJ_{\mrE}$ are diagonal or close to diagonal.
	
	\subsection{Supporting Lemmas}
	Define norms $\|\cdot\|_{\mrS}$ and $\|\cdot\|_{\mrD}$ such that for all $\mx\in\mathbb{R}^n$, $\|\mx\|_{\mrS}:=\|\mS\mx\|_2$ and $\|\mx\|_{\mrD}:=\|\mD\mx\|_2$. Correspondingly, for any matrix $\mW\in\mathbb{R}^{n\times n}$, its matrix norms are given by $\|\mW\|_\mrS:=\|\mS^{-1}\mW\mS\|_2$ and $\|\mW\|_\mrD:=\|\mD^{-1}\mW\mD\|_2$, respectively.
	Denote 
	$\tilde{\mT}_k:=\mT_k-\bar{\mT}$, $\tilde{\mE}_k:=\mE_k-\bar{\mE}$, and $\tilde{\mR}_k:=\mR_k-\bar{\mR}=\gamma\tilde{\mT}_k$, $\tilde{\mC}_k:=\mC_k-\bar{\mC}=\gamma\tilde{\mE}_k$. 
	We have a few supporting lemmas.
	\begin{lemma}
		\label{lem: x_consensus gpp}
		Under Assumption \ref{asp: nonempty root set}, we have
		\begin{equation*}
		\bE\left[\left\|\left(\mI-\frac{\mathbf{1}\bar{u}^{\T}}{n}\right)\mR_k\mx_k\right\|_{\mrS}^2\mid \mx_k\right]
		\le \sigma_{\bar{\mrR}}\left\|\mx_k-\mathbf{1}\ox_k\right\|_{\mrS}^2,
		\end{equation*}
		where $\sigma_{\bar{\mrR}}:=\|\mJ_{\mrR}\|_2^2+\gamma^2\|\bar{\mV}_{\mrT}\|_2$ with
		\begin{equation*}
		\bar{\mV}_{\mrT}:=\bE\left[(\mS^{-1})^{\T}\tilde{\mT}_k^{\T}\left(\mI-\frac{\bar{u}\mathbf{1}^{\T}}{n}\right)\mS^{\T}\mS\left(\mI-\frac{\mathbf{1}\bar{u}^{\T}}{n}\right)\tilde{\mT}_k\mS^{-1}\right],
		\end{equation*}
		and
		\begin{equation*}
		\bE\left[\left\|\left(\mI-\frac{\mathbf{1}\bar{u}^{\T}}{n}\right)\mQ_k\my_k\right\|_{\mrS}^2\mid\mathcal{H}_k\right]
		\le \|\bar{\mV}_{\mrQ}\|_2\|\my_k\|_{\mrS}^2,
		\end{equation*}
		where 
		\begin{equation*}
		\bar{\mV}_{\mrQ}:=\bE\left[(\mS^{-1})^{\T}\mQ_k^{\T}\left(\mI-\frac{\mathbf{1}\bar{u}^{\T}}{n}\right)^{\T}\mS^{\T}\mS\left(\mI-\frac{\mathbf{1}\bar{u}^{\T}}{n}\right)\mQ_k\mS^{-1}\right].
		\end{equation*}
		In particular, there exist $\bar{\gamma}_{\mrR}>0$ such that for all $\gamma\in(0,\bar{\gamma}_{\mrR})$, we have $\sigma_{\bar{\mrR}}<1$.
	\end{lemma}
	\begin{proof}
		By definition,
		\begin{align}
		\label{x_k+1-ux_k+1_pre}
		\begin{array}{l}
		\quad\left\|\left(\mI-\frac{\mathbf{1}\bar{u}^{\T}}{n}\right)\mR_k\mx_k\right\|_{\mrS}^2\\
		=\trace\left\{\left[\mS\left(\mI-\frac{\mathbf{1}\bar{u}^{\T}}{n}\right)\mR_k\mx_k\right]^{\T}\mS\left(\mI-\frac{\mathbf{1}\bar{u}^{\T}}{n}\right)\mR_k\mx_k\right\}\\
		=\trace\left\{\mx_k^{\T}\mR_k^{\T}\left(\mI-\frac{\bar{u}\mathbf{1}^{\T}}{n}\right)\mS^{\T}\mS\left(\mI-\frac{\mathbf{1}\bar{u}^{\T}}{n}\right)\mR_k\mx_k\right\}.
		\end{array}
		\end{align}
		In what follows, we omit the symbol $\trace\{\cdot\}$ to simplify notation; all matrices $\mW\in\mathbb{R}^{p\times p}$ refer to $\trace\{\mW\}$.
		The readers may conveniently assume $p=1$ in the following.
		
		Let $\mathcal{H}_k$ denote the history $\{\mx_1,\mx_2,\ldots,\mx_k,\my_k\}$. Taking conditional expectation on both sides of (\ref{x_k+1-ux_k+1_pre}), we have
		{\small\begin{align}
			\label{x_k+1-ux_k+1_expectation_pre}
			\begin{array}{l}
			\quad\bE\left[\left\|\left(\mI-\frac{\mathbf{1}\bar{u}^{\T}}{n}\right)\mR_k\mx_k\right\|_{\mrS}^2\mid\mathcal{H}_k\right]\\
			= \mx_k^{\T}\bar{\mR}^{\T}\left(\mI-\frac{\bar{u}\mathbf{1}^{\T}}{n}\right)\mS^{\T}\mS\left(\mI-\frac{\mathbf{1}\bar{u}^{\T}}{n}\right)\bar{\mR}\mx_k\\
			\quad+\bE\left[\mx_k^{\T}\tilde{\mR}_k^{\T}\left(\mI-\frac{\bar{u}\mathbf{1}^{\T}}{n}\right)\mS^{\T}\mS\left(\mI-\frac{\mathbf{1}\bar{u}^{\T}}{n}\right)\tilde{\mR}_k\mx_k\mid\mathcal{H}_k\right]\\
			= \left(\mx_k-\mathbf{1}\ox_k\right)^{\T}\left(\bar{\mR}-\frac{\mathbf{1}\bar{u}^{\T}}{n}\right)^{\T}\mS^{\T}\mS\left(\bar{\mR}-\frac{\mathbf{1}\bar{u}^{\T}}{n}\right)\left(\mx_k-\mathbf{1}\ox_k\right)\\
			\quad+\gamma^2\left(\mx_k-\mathbf{1}\ox_k\right)^{\T}\mS^{\T}\bE\left[(\mS^{-1})^{\T}\tilde{\mT}_k^{\T}\left(\mI-\frac{\bar{u}\mathbf{1}^{\T}}{n}\right)\mS^{\T}\right.\\
			\quad\left.\cdot\mS\left(\mI-\frac{\mathbf{1}\bar{u}^{\T}}{n}\right)\tilde{\mT}_k\mS^{-1}\right]\mS\left(\mx_k-\mathbf{1}\ox_k\right).
			\end{array}
			\end{align}}\normalsize
		In light of Lemma \ref{lem: Jordan form RC} (see \cite{horn1990matrix}),
		$\left\|\bar{\mR}-\frac{\mathbf{1}\bar{u}^{\T}}{n}\right\|_{\mrS}=\left\|\mS\left(\bar{\mR}-\frac{\mathbf{1}\bar{u}^{\T}}{n}\right)\mS^{-1}\right\|_2= \|\mJ_{\mrR}\|_2$.
		Then from (\ref{x_k+1-ux_k+1_expectation_pre}) and the definition of $\bar{\mT}$, we have
		\begin{align*}
		\begin{array}{l}
		\quad\bE\left[\left\|\left(\mI-\frac{\mathbf{1}\bar{u}^{\T}}{n}\right)\mR_k\mx_k\right\|_{\mrS}^2\mid\mathcal{H}_k\right]\\
		\le \left\|\bar{\mR}-\frac{\mathbf{1}\bar{u}^{\T}}{n}\right\|_{\mrS}^2\left\|\mx_k-\mathbf{1}\ox_k\right\|_{\mrS}^2
		+\gamma^2\|\bar{\mV}_{\mrT}\|_2\left\|\mx_k-\mathbf{1}\ox_k\right\|_{\mrS}^2\\
		= \left(\|\mJ_{\mrR}\|_2^2+\gamma^2\|\bar{\mV}_{\mrT}\|_2\right)\left\|\mx_k-\mathbf{1}\ox_k\right\|_{\mrS}^2.
		\end{array}
		\end{align*}
		
		The second relation follows from
		\begin{align*}
		\begin{array}{l}
		\quad\bE\left[\left\|\left(\mI-\frac{\mathbf{1}\bar{u}^{\T}}{n}\right)\mQ_k\my_k\right\|_{\mrS}^2\mid\mathcal{H}_k\right]\\
		=\bE\left[\left\|\my_k^{\T}\mS^{\T}(\mS^{-1})^{\T}\mQ_k^{\T}\left(\mI-\frac{\mathbf{1}\bar{u}^{\T}}{n}\right)^{\T}\mS^{\T}\mS\left(\mI-\frac{\mathbf{1}\bar{u}^{\T}}{n}\right)\right.\right.\\
		\quad\left.\left.\cdot\mQ_k\mS^{-1}\mS\my_k\right\|_2\mid\mathcal{H}_k\right]\le \|\bar{\mV}_{\mrQ}\|_2\|\my_k\|_{\mrS}^2.
		\end{array}
		\end{align*}
		
		Given the definition of $\mJ_{\mrR}$ and Lemma \ref{lem: eigenvalues_bar_T E}, we know $\|\mJ_{\mrR}\|=1-\mathcal{O}(\gamma)$. Hence when $\gamma$ is sufficiently small, we have $\sigma_{\bar{\mrR}}=\|\mJ_{\mrR}\|_2^2+\gamma^2\|\bar{\mV}_{\mrT}\|_2<1$.
	\end{proof}
	\begin{lemma}
		\label{lem: y_alignment gpp}
		Under Assumption \ref{asp: nonempty root set}, we have
		\begin{multline*}
		\bE\left[\left\|\left(\mI-\frac{\bar{v} \mathbf{1}^{\T}}{n}\right)\mC_k\my_k\right\|_{\mrD}^2\mid\mathcal{H}_k\right]
		\le \sigma_{\bar{\mrC}}\left\|\my_k-\bar{v}\oy_k\right\|_{\mrD}^2\\
		+2\gamma^2\|\bar{\mV}_{\mrE}\|_2\left\|\bar{v}\oy_k\right\|_{\mrD}^2,
		\end{multline*}
		where $\sigma_{\bar{\mrC}}:= \|\mJ_{\mrC}\|_2^2+2\gamma^2\|\bar{\mV}_{\mrE}\|_2$ with $\bar{\mV}_{\mrE} :=\bE\left[(\mD^{-1})^{\T}\tilde{\mE}_k^{\T}\mD^{\T}\mD\tilde{\mE}_k\mD^{-1}\right]$.
		In particular, there exist $\bar{\gamma}_{\mrC}>0$ such that for all $\gamma\in(0,\bar{\gamma}_{\mrC})$, we have $\sigma_{\bar{\mrC}}<1$.
	\end{lemma}
	\begin{proof}
		Note that
		\begin{align*}
		\begin{array}{l}
		\quad\left\|\left(\mI-\frac{\bar{v} \mathbf{1}^{\T}}{n}\right)\mC_k\my_k\right\|_{\mrD}^2\\
		=\left[\mD\left(\mI-\frac{\bar{v} \mathbf{1}^{\T}}{n}\right)\mC_k\my_k\right]^{\T}\mD\left(\mI-\frac{\bar{v} \mathbf{1}^{\T}}{n}\right)\mC_k\my_k\\
		=\my_k^{\T}\mC_k^{\T}\left(\mI-\frac{\mathbf{1}\bar{v}^{\T}}{n}\right)\mD^{\T}\mD\left(\mI-\frac{\bar{v} \mathbf{1}^{\T}}{n}\right)\mC_k\my_k.
		\end{array}
		\end{align*}
		Taking conditional expectation on both sides,
		\begin{align*}
		\begin{array}{l}
		\quad\bE\left[\left\|\left(\mI-\frac{\bar{v} \mathbf{1}^{\T}}{n}\right)\mC_k\my_k\right\|_{\mrD}^2\mid\mathcal{H}_k\right]\\
		=\my_k^{\T}\bar{\mC}^{\T}\left(\mI-\frac{\mathbf{1}\bar{v}^{\T}}{n}\right)\mD^{\T}\mD\left(\mI-\frac{\bar{v} \mathbf{1}^{\T}}{n}\right)\bar{\mC}\my_k\\
		\quad+\bE\left[\my_k^{\T}\tilde{\mC}_k^{\T}\left(\mI-\frac{\mathbf{1}\bar{v}^{\T}}{n}\right)\mD^{\T}\mD\left(\mI-\frac{\bar{v} \mathbf{1}^{\T}}{n}\right)\tilde{\mC}_k\my_k\mid\mathcal{H}_k\right]\\
		=\left(\my_k-\bar{v}\oy_k\right)^{\T}\left(\bar{\mC}^{\T}-\frac{\mathbf{1}\bar{v}^{\T}}{n}\right)\mD^{\T}\mD\left(\bar{\mC}-\frac{\bar{v} \mathbf{1}^{\T}}{n}\right)\left(\my_k-\bar{v}\oy_k\right)\\
		\quad+\my_k^{\T}\bE\left[\tilde{\mC}_k^{\T}\mD^{\T}\mD\tilde{\mC}_k\right]\my_k\\
		\le \left\|\bar{\mC}-\frac{\bar{v} \mathbf{1}^{\T}}{n}\right\|_{\mrD}^2\left\|\my_k-\bar{v}\oy_k\right\|_{\mrD}^2\\
		\quad+\gamma^2\my_k^{\T}\mD^{\T}\bE\left[(\mD^{-1})^{\T}\tilde{\mE}_k^{\T}\mD^{\T}\mD\tilde{\mE}_k\mD^{-1}\right]\mD\my_k\\
		\le \|\mJ_{\mrC}\|_2^2\left\|\my_k-\bar{v}\oy_k\right\|_{\mrD}^2+\gamma^2\|\bar{\mV}_{\mrE}\|_2\left\|\my_k\right\|_{\mrD}^2\\
		\le \left(\|\mJ_{\mrC}\|_2^2+2\gamma^2\|\bar{\mV}_{\mrE}\|_2\right)\left\|\my_k-\bar{v}\oy_k\right\|_{\mrD}^2+2\gamma^2\|\bar{\mV}_{\mrE}\|_2\left\|\bar{v}\oy_k\right\|_{\mrD}^2,
		\end{array}
		\end{align*}
		where we invoked Lemma \ref{lem: Jordan form RC} for the second to last relation.
	\end{proof}
	
	\begin{lemma}
		\label{lem: x-xstar_pre gpp}
		Under Assumption \ref{asp: nonempty root set}, we have
		\begin{align*}
		\begin{array}{l}
		\quad\bE\left[\|\bar{x}_{k+1}-x^*\|_2^2\mid\mathcal{H}_k\right]\\
		\le \left[1-\alpha\eta\mu+\frac{4\alpha^2}{n^2}\left\|\bE\left[\mQ_k^{\T}\bar{u}\bar{u}^{\T}\mQ_k\right]\right\|_2\|\bar{v}\|_2^2L^2\right]\|\bar{x}_k-x^*\|_2^2\\
		\quad+\left(\frac{2\alpha\eta L^2}{\mu n}+\frac{4\alpha^2 L^2}{n^3}\left\|\bE\left[\mQ_k^{\T}\bar{u}\bar{u}^{\T}\mQ_k\right]\right\|_2\|\bar{v}\|_2^2\right.\\
		\quad\left.+\frac{1}{n^2}\left\|\bE\left[\tilde{\mR}_k^{\T}\bar{u}\bar{u}^{\T}\tilde{\mR}_k\right]\right\|_2\right)\|\mx_k-\mathbf{1}\ox_k\|_2^2\\
		\quad+\left(\frac{2\alpha}{\eta \mu n^2}\left\|\bar{u}^{\T}\bE[\mQ_k]\right\|_2^2+\frac{2\alpha^2}{{n^2}}\left\|\bE\left[\mQ_k^{\T}\bar{u}\bar{u}^{\T}\mQ_k\right]\right\|_2\right)\\
		\quad\cdot\|\my_k-\bar{v}\oy_k\|_2^2,
		\end{array}
		\end{align*}
		where $\eta := \bar{u}^{\T}\bE[\mQ_k] \bar{v}/n>0.$
	\end{lemma}
	\begin{proof}
		See Appendix \ref{proof: lem: x-xstar_pre gpp}.
	\end{proof}
	Similar to Lemma \ref{lem: norm equivalence}, we have the following relation between norms $\|\cdot\|_2$, $\|\cdot\|_{\mrS}$ and $\|\cdot\|_{\mrD}$.
	\begin{lemma}
		\label{lem: norm equivalence 2}
		There exist constants $\delta_{\mrD,\mrS}, \delta_{\mrD,2}, \delta_{\mrS,\mrD}, \delta_{\mrS,2}>0$ such that for all $\mx\in\mathbb{R}^{n\times p}$, we have  $\|\mx\|_{\mrD}\le \delta_{\mrD,\mrS}\|\mx\|_{\mrS}$, $\|\mx\|_{\mrD}\le \delta_{\mrD,2}\|\mx\|_2$, $\|\mx\|_{\mrS}\le \delta_{\mrS,\mrD}\|\mx\|_{\mrD}$,  and $\|\mx\|_{\mrS}\le \delta_{\mrS,2}\|\mx\|_2$. In addition, with 
		a proper rescaling of the norms $\|\cdot\|_{\mrS}$ and $\|\cdot\|_{\mrD}$, we have
		$\|\mx\|_2\le \|\mx\|_{\mrS}$ and $\|\mx\|_2\le \|\mx\|_{\mrD}$ for all $\mx$.
	\end{lemma}
	
	In the following lemma, we establish a linear system of inequalities that bound $\bE[\|\ox_{k+1}-x^*\|_2^2]$, $\bE[\|\mx_{k+1}-\mathbf{1}\ox_k\|_{\mrS}^2]$ and $\bE[\|\my_{k+1}-v\oy_k\|_{\mrD}^2]$.
	\begin{lemma}
		\label{lem: linear system gossip}
		Under Assumptions \ref{asp; strconvex Lipschitz}-\ref{asp: nonempty root set}, we have the following linear system of inequalities:
		\begin{equation}
		\label{main ineqalities_gpp}
		\begin{array}{l}
		\begin{bmatrix}
		\bE[\|\ox_{k+1}-x^*\|_2^2]\\
		\bE[\|\mx_{k+1}-\mathbf{1}\ox_{k+1}\|_{\mrS}^2]\\
		\bE[\|\my_{k+1}-v\oy_{k+1}\|_{\mrD}^2]
		\end{bmatrix}
		\le
		\mB
		\begin{bmatrix}
		\bE[\|\ox_k-x^*\|_2^2]\\
		\bE[\|\mx_k-\mathbf{1}\ox_k\|_{\mrS}^2]\\
		\bE[\|\my_k-v\oy_k\|_{\mrD}^2]
		\end{bmatrix},
		\end{array}
		\end{equation}
		where the inequality is to be taken component-wise, and elements of the transition matrix $\mB=[b_{ij}]$ are given by:
		\begin{align}
		\label{matrix_B}
		\begin{array}{l}
		b_{11}=1-\alpha\eta\mu+\frac{4\alpha^2 L^2}{n^2}\left\|\bE\left[\mQ_k^{\T}\bar{u}\bar{u}^{\T}\mQ_k\right]\right\|_2\|\bar{v}\|_2^2,\\
		b_{21}=3\alpha^2L^2\left(\frac{1+\sigma_{\bar{\mrR}}}{1-\sigma_{\bar{\mrR}}}\right)\|\bar{\mV}_{\mrQ}\|_2\|\bar{v}\|_{\mrS}^2,\\
		b_{31}=2L^2\left[\frac{(1+\sigma_{\bar{\mrC}})}{\sigma_{\bar{\mrC}}}\gamma^2\|\bar{\mV}_{\mrE}\|_2\|\bar{v}\|_{\mrD}^2\right.\\
		\qquad\quad\left.+4\alpha^2\delta_{\mrD,2}^2 L^2\left(\frac{1+\sigma_{\bar{\mrC}}}{1-\sigma_{\bar{\mrC}}}\right)\left\|\mI-\frac{\bar{v} \mathbf{1}^{\T}}{n}\right\|_{\mrD}^2\bE[\|\mQ_k\|_2^2]\|\bar{v}\|_2^2\right],\\
		b_{12}=\frac{2\alpha\eta L^2}{\mu n}+\frac{4\alpha^2 L^2}{n^3}\left\|\bE\left[\mQ_k^{\T}\bar{u}\bar{u}^{\T}\mQ_k\right]\right\|_2\|\bar{v}\|_2^2\\
		\qquad\quad+\frac{1}{n^2}\left\|\bE\left[\tilde{\mR}_k^{\T}\bar{u}\bar{u}^{\T}\tilde{\mR}_k\right]\right\|_2,\\
		b_{22}=\frac{(1+\sigma_{\bar{\mrR}})}{2}+\frac{3\alpha^2L^2}{n}\left(\frac{1+\sigma_{\bar{\mrR}}}{1-\sigma_{\bar{\mrR}}}\right)\|\bar{\mV}_{\mrQ}\|_2\|\bar{v}\|_{\mrS}^2,\\
		b_{32}=2\delta_{\mrD,2}^2 L^2\left(\frac{1+\sigma_{\bar{\mrC}}}{1-\sigma_{\bar{\mrC}}}\right)\left\|\mI-\frac{\bar{v} \mathbf{1}^{\T}}{n}\right\|_{\mrD}^2\\ \qquad\quad\cdot\left[\bE\left[\|\mR_k-\mI\|_2^2\right]+\frac{4\alpha^2 L^2}{n}\bE[\|\mQ_k\|_2^2]\|\bar{v}\|_2^2\right]\\
		\qquad\quad+\frac{2\gamma^2L^2}{n}\frac{(1+\sigma_{\bar{\mrC}})}{\sigma_{\bar{\mrC}}}\|\bar{\mV}_{\mrE}\|_2\|\bar{v}\|_{\mrD}^2,\\
		b_{13}=\frac{2\alpha}{\eta \mu n^2}\left\|\bar{u}^{\T}\bE[\mQ_k]\right\|_2^2+\frac{2\alpha^2}{{n^2}}\left\|\bE\left[\mQ_k^{\T}\bar{u}\bar{u}^{\T}\mQ_k\right]\right\|_2,\\
		b_{23}=3\alpha^2\left(\frac{1+\sigma_{\bar{\mrR}}}{1-\sigma_{\bar{\mrR}}}\right)\|\bar{\mV}_{\mrQ}\|_2\delta_{\mrS,\mrD},\\
		b_{33}=\frac{(1+\sigma_{\bar{\mrC}})}{2}+4\alpha^2 \delta_{\mrD,2}^2 L^2\left(\frac{1+\sigma_{\bar{\mrC}}}{1-\sigma_{\bar{\mrC}}}\right)\left\|\mI-\frac{\bar{v} \mathbf{1}^{\T}}{n}\right\|_{\mrD}^2 \bE[\|\mQ_k\|_2^2].
		\end{array}
		\end{align}
	\end{lemma}
	\begin{proof}
		See Appendix \ref{proof: lem: linear system gossip}.
	\end{proof}
	
	Now we are ready to prove the main convergence result for G-Push-Pull.
	\subsection{Proof of Theorem \ref{theory:main_gpp}}
	Let
	\begin{align}
	\label{d1d7}
	\begin{array}{l}
	d_1:=\frac{2\left\|\bE\left[\mQ_k^{\T}\bar{u}\bar{u}^{\T}\mQ_k\right]\right\|_2}{n^2}, \\
	d_2:=3\left(\frac{1+\sigma_{\bar{\mrR}}}{1-\sigma_{\bar{\mrR}}}\right)\|\bar{\mV}_{\mrQ}\|_2\\
	d_3:=4\delta_{\mrD,2}^2 \left(\frac{1+\sigma_{\bar{\mrC}}}{1-\sigma_{\bar{\mrC}}}\right)\left\|\mI-\frac{\bar{v} \mathbf{1}^{\T}}{n}\right\|_{\mrD}^2 \bE[\|\mQ_k\|_2^2], \\
	d_4:= \frac{1}{n}\frac{(1+\sigma_{\bar{\mrC}})}{\sigma_{\bar{\mrC}}}\|\bar{\mV}_{\mrE}\|_2\|\bar{v}\|_{\mrD}^2\\
	d_5:= \frac{1}{n^2}\left\|\bE\left[\tilde{\mR}_k^{\T}\bar{u}\bar{u}^{\T}\tilde{\mR}_k\right]\right\|_2, \\
	d_6:= \delta_{\mrD,2}^2 \left(\frac{1+\sigma_{\bar{\mrC}}}{1-\sigma_{\bar{\mrC}}}\right)\left\|\mI-\frac{\bar{v} \mathbf{1}^{\T}}{n}\right\|_{\mrD}^2 \bE[\|\mR_k-\mI\|_2^2]\\
	d_7:= \frac{\left\|\bar{u}^{\T}\bE[\mQ_k]\right\|_2^2}{n^2}.
	\end{array}
	\end{align}
	We can rewrite the elements of $\mB$ as
	\begin{align*}
	\begin{array}{l}
	b_{11} = 1-\alpha\eta\mu+2\alpha^2 L^2\|\bar{v}\|_2^2d_1,\\
	b_{21} = \alpha^2L^2 d_2\|\bar{v}\|_{\mrS}^2,\\
	b_{31} = 2L^2\left(\gamma^2 d_4+\alpha^2 L^2 d_3\|\bar{v}\|_2^2\right),\\
	b_{12} = \frac{2\alpha\eta L^2}{\mu n}+\frac{2\alpha^2 L^2 \|\bar{v}\|_2^2 d_1}{n}+d_5,\\
	b_{22} = \frac{(1+\sigma_{\bar{\mrR}})}{2}+\frac{\alpha^2L^2 d_2\|\bar{v}\|_{\mrS}^2}{n},\\
	b_{32} = 2 L^2 (d_6+\alpha^2 L^2 d_3\|\bar{v}\|_2^2+\gamma^2 d_4),\\
	b_{13} = \frac{2\alpha}{\eta \mu }d_7+\alpha^2 d_1,\\
	b_{23} = \alpha^2 d_2\delta_{\mrS,\mrD},\\
	b_{33} = \frac{(1+\sigma_{\bar{\mrC}})}{2}+\alpha^2 L^2 d_3.
	\end{array}
	\end{align*}
	According to Lemma \ref{lem: spectral radii}, a sufficient condition for $\rho(\mB)<1$ is $b_{11},b_{22},b_{33}<1$ and $\det(\mI-\mB)>0$, or
	\begin{align} \label{|I-B|>0}
	\begin{array}{l}
	\mathrm{det}(\mathbf{I}-\mB)=(1-b_{11})(1-b_{22})(1-b_{33})-b_{12}b_{23}b_{31}\\
	-b_{13}b_{21}b_{32}-(1-b_{22})b_{13}b_{31}-(1-b_{11})b_{23}b_{32}\\
	-(1-b_{33})b_{12}b_{21}>0.
	\end{array}
	\end{align}
	Let $\alpha$ satisfy the following inequalities.
	\begin{align}\label{alpha_conditions_pre_gpp}
	\begin{array}{l}
	b_{11}\le 1-\frac{1}{2}\alpha\eta\mu,\qquad b_{22}\le \frac{(3+\sigma_{\bar{\mrR}})}{4},\\
	b_{33}\le \frac{(3+\sigma_{\bar{\mrC}})}{4}, \qquad \frac{2\alpha^2 L^2\|\bar{v}\|_2^2 d_1}{n}\le \frac{\alpha\eta L^2}{\mu n},\\
	\alpha^2 d_1\le \frac{\alpha}{\eta \mu }d_7,\qquad \alpha^2 L^2 d_3\|\bar{v}\|_2^2\le d_6.
	\end{array}
	\end{align}
	Then it is sufficient that
	\begin{align*}
	\begin{array}{l}
	\frac{1}{2}\alpha\eta\mu \frac{(1-\sigma_{\bar{\mrR}})}{4}\frac{(1-\sigma_{\bar{\mrC}})}{4}\\
	-\left(\frac{3\alpha\eta L^2}{\mu n}+d_5\right)\alpha^2 d_2\delta_{\mrS,\mrD}
	2L^2\left(\gamma^2 d_4+\alpha^2 L^2 d_3\|\bar{v}\|_2^2\right)\\
	-\frac{3\alpha}{\eta \mu }d_7 	\alpha^2L^2 d_2\|\bar{v}\|_{\mrS}^2 2 L^2 \left(d_6+\alpha^2 L^2 d_3\|\bar{v}\|_2^2+\gamma^2 d_4\right)\\
	-\frac{(1-\sigma_{\bar{\mrR}})}{4}\frac{3\alpha}{\eta \mu }d_7 2L^2\left(\gamma^2 d_4+\alpha^2 L^2 d_3\|\bar{v}\|_2^2\right)\\
	-\frac{1}{2}\alpha\eta\mu \alpha^2 d_2\delta_{\mrS,\mrD} 2 L^2 \left(d_6+\alpha^2 L^2 d_3\|\bar{v}\|_2^2+\gamma^2 d_4\right)\\
	-\frac{(1-\sigma_{\bar{\mrC}})}{4}\left(\frac{3\alpha\eta L^2}{\mu n}+d_5\right)\alpha^2L^2 d_2\|\bar{v}\|_{\mrS}^2>0.
	\end{array}
	\end{align*}
	Since $\alpha^2 L^2 d_3\|\bar{v}\|_2^2\le d_6$ from (\ref{alpha_conditions_pre_gpp}), we only need
	\begin{align*}
	\begin{array}{l}
	\frac{1}{32}\eta\mu (1-\sigma_{\bar{\mrR}})(1-\sigma_{\bar{\mrC}})\\
	-\left(\frac{3\alpha\eta L^2}{\mu n}+d_5\right)2\alpha L^2d_2\delta_{\mrS,\mrD}
	\left(\gamma^2 d_4+d_6\right)\\
	-\frac{6\alpha^2 L^4}{\eta \mu }d_2 d_7\|\bar{v}\|_{\mrS}^2 \left(2d_6+\gamma^2 d_4\right)\\
	-\frac{3(1-\sigma_{\bar{\mrR}})}{2}\frac{L^2 d_7}{\eta \mu }\left(\gamma^2 d_4+\alpha^2 L^2 d_3\|\bar{v}\|_2^2\right)\\
	-\eta\mu \alpha^2 L^2 d_2\delta_{\mrS,\mrD} \left(2d_6+\gamma^2 d_4\right)\\
	-\frac{(1-\sigma_{\bar{\mrC}})}{4}\left(\frac{3\alpha\eta L^2}{\mu n}+d_5\right)\alpha L^2 d_2\|\bar{v}\|_{\mrS}^2>0.
	\end{array}
	\end{align*}
	We can rewrite the above inequality as 
	$
	c_4\alpha^2+c_5\alpha-c_6<0
	$,
	where
	\begin{align}
	\label{c4}
	\begin{array}{l}
	c_4=\frac{6\eta L^4}{\mu n} d_2\delta_{\mrS,\mrD}\left(d_6+\gamma^2 d_4\right)+\frac{6 L^4}{\eta \mu }d_2 d_7\|\bar{v}\|_{\mrS}^2 \left(2d_6+\gamma^2 d_4\right)\\
	\qquad+\frac{3(1-\sigma_{\bar{\mrR}})}{2}\frac{L^4 d_3 d_7}{\eta \mu }\|\bar{v}\|_2^2+\eta\mu L^2 d_2\delta_{\mrS,\mrD} \left(2d_6+\gamma^2 d_4\right)\\
	\qquad+\frac{3(1-\sigma_{\bar{\mrC}})}{4}\frac{\eta L^4 d_2}{\mu n}\|\bar{v}\|_{\mrS}^2,
	\end{array}
	\end{align}
	\begin{equation}
	\label{c5}
	\begin{array}{l}
	c_5=2L^2 d_2  d_5\delta_{\mrS,\mrD}\left(d_6+\gamma^2 d_4\right)+\frac{(1-\sigma_{\bar{\mrC}})}{4}L^2 d_2  d_5\|\bar{v}\|_{\mrS}^2,
	\end{array}
	\end{equation}
	and 
	\begin{equation}
	\label{c6}
	\begin{array}{l}
	c_6=\frac{1}{32}\eta\mu (1-\sigma_{\bar{\mrR}})(1-\sigma_{\bar{\mrC}})-\frac{3(1-\sigma_{\bar{\mrR}})}{2}\frac{L^2 d_4 d_7}{\eta \mu }\gamma^2.
	\end{array}
	\end{equation}
	In light of \eqref{gamma_condition_gpp}, we have $c_6>0$. Then
	\begin{equation*}
	\alpha \le \frac{2c_6}{c_5+\sqrt{c_5^2+4c_4c_6}}
	\end{equation*}
	is sufficient.
	
	\section{SIMULATIONS}
	\label{sec: simulation}
	In this section, we provide numerical comparisons of a few different algorithms under both synchronous and asynchronous random-gossip settings. The problem we consider is sensor fusion over a network, which is similar to the one considered in~\cite{xu2017convergence}. The estimation problem can be described as
	\[
	\min\limits_{x\in\R^p}~\sum\limits_{i=1}^n\left(\left\|z_i-H_ix\right\|^2+\lambda_i\left\|x\right\|^2\right),
	\]
	where $x$ is the unknown parameter to be estimated, $H_i\in\R^{s\times p}$ and $z_i\in\mathbb{R}^s$ represent the measurement matrix and the noisy observation of sensor $i$, respectively, and $\lambda_i$ is the regularization parameter for the local cost function of sensor $i$.
	
	We consider a sensor network that is randomly generated in a unit square, and two sensors are connected within a certain sensing range. The sensors are assumed to have asymmetric sensing ranges so as to construct a directed network, i.e., $1.5\sqrt{\log n/n}$ for outgoing links and $0.75\sqrt{\log n/n}$ for incoming links. We set $n=20,\,p=20,\,s=1$ and $\lambda_i=0.01,\forall i\in\mathcal{N}$ so that each local cost function is ill-conditioned, necessitating coordination among agents to achieve fast convergence. The measurement matrix is randomly generated from a standard normal distribution which is then normalized such that its Lipschitz constant is equal to $1$. We design the weight matrices $\mathbf{R}$ and $\mathbf{C}$ based on the same underlying graph (i.e., $\Gra_\mathbf{R}=\Gra_\mathbf{C}=\Gra$) and the constant weight rule, i.e., $\mR=\mI-\frac{1}{2d_{\max}^{\mathrm{in}}}\mathbf{L}_{\mrR}$ where $d_{\max}^{\mathrm{in}}$ is the maximum in-degree, and $\mathbf{L}_{\mrR}$ defines the Laplacian matrix corresponding to $\Gra_\mathbf{R}$ (using in-degree). Similarly, $\mC=\mI-\frac{1}{2d_{\max}^{\mathrm{out}}}\mathbf{L}_{\mrC}$.
	
	We compare our proposed Push-Pull algorithm against Push-DIGing \cite{nedic2017achieving} and Xi-Row \cite{xi2018linear} that are also applicable to directed networks. Push-DIGing is an algorithm building on push-sum protocols which works on merely column-stochastic matrices and thus only need push operations for the information dissemination in the network. Xi-Row is an algorithm that only uses row-stochastic matrices and thus only requires pull operations to fetch information in the network. By contrast, the proposed Push-Pull algorithm uses both row-stochastic and column-stochastic matrices for the information diffusion. 
	As we will show shortly in our simulations, Push-Pull works much better especially for ill-conditioned problems and when graphs are not well balanced. This is due to the fact that nodes with very few in-degree (resp., out-degree) neighbors (e.g., in significantly unbalanced directed graphs) will become a bottleneck to the information flow for row-stochastic (resp., column-stochastic) weight matrices. In contrast, Push-Pull relies on both out-degree and in-degree neighbors of the network and can thus diffuse information much fast. It should be also noted that the per-node storage complexity of Push-Pull (or Push-DIGing) is $O(p)$ while that of Xi-Row is $O(n+p)$. Since, at each iteration, the amount of data transmitted over each link also scales at such orders for these algorithms, respectively. For large-scale networks ($n\gg p$), Xi-Row may suffer from high needs in storage/bandwidth compared to the other methods.
	
	Fig.~\ref{fig:fixnet} illustrates the performance of the above algorithms under a randomly generated fixed network in terms of the (normalized) residual $\frac{\|\bx_k-\bx^*\|_2^2}{\|\bx_0-\bx^*\|_2^2}$. Since the upperbound on the stepsize for all the algorithms are derived from the small gain theorem (or similar techniques) and can be very conservative, we hand-optimize the stepsize for each method to make a fair comparison.  It can be seen from Fig.~\ref{fig:fixnet} that Push-Pull allows for much larger value of the stepsize compared to Push-DIGing and Xi-Row. In addition, it also enjoys much faster convergence speed.
	
	Under the asynchronous random-gossip setting, we compare G-Push-Pull (see section \ref{sec: G-Push-Pull}) against a variant of Push-DIGing which is shown to be applicable to the time-varying scenario~\cite{nedic2017achieving}. We use a random link activation model, that is, at each iteration, a subset of links will be randomly activated following a certain Bernoulli process ${B(0.1)}$, and the associated nodes will be awakened via ``push-notification'' or ``pull-notification''. These awakened nodes will then communicate with each other according to the gossip protocol that is designed based on the aforementioned constant weight rule.  The simulation results are averaged over 20 runs. Fig.~\ref{fig:asynet} illustrates the performance of the proposed G-Push-Pull algorithm against Push-DIGing. It can be seen that Push-DIGIng allows for very small values of the stepsize and suffers from some ``spikes" due to the use of division operations in the algorithm (note that the divisors can scale badly at the order of $\Omega(n^{-n})$~\cite{nedic2017achieving}).  In contrast, G-Push-Pull allows for very large value of stepsize and enjoys much faster convergence speed. More importantly, it converges linearly and steadily to the optimal solution.
	
	\begin{figure}[h]
		\begin{center}
			\vspace{-1em}
			\subfloat[Fixed directed network]{
				\includegraphics
				[width=0.9\linewidth]{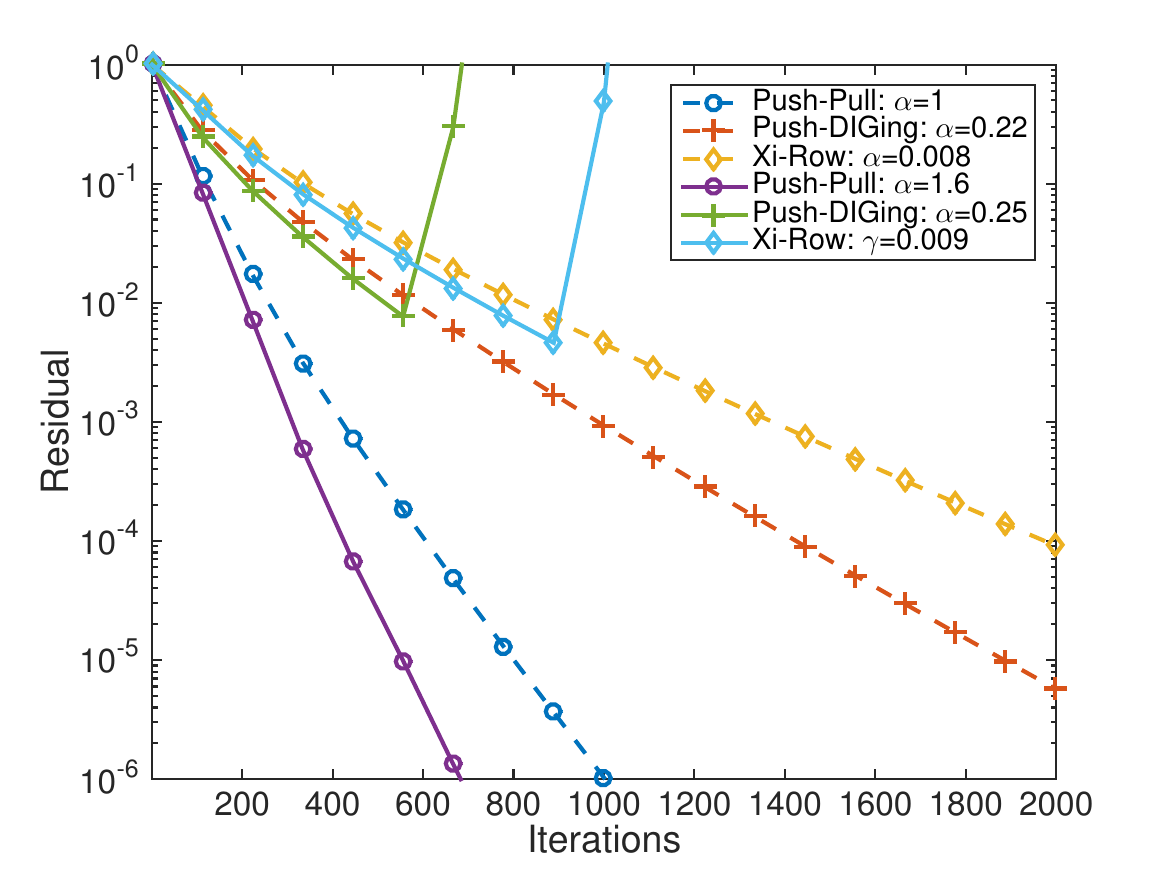}\label{fig:fixnet}}
			\\
			\subfloat[Asynchronous directed network]{
				\includegraphics
				[width=0.9\linewidth]{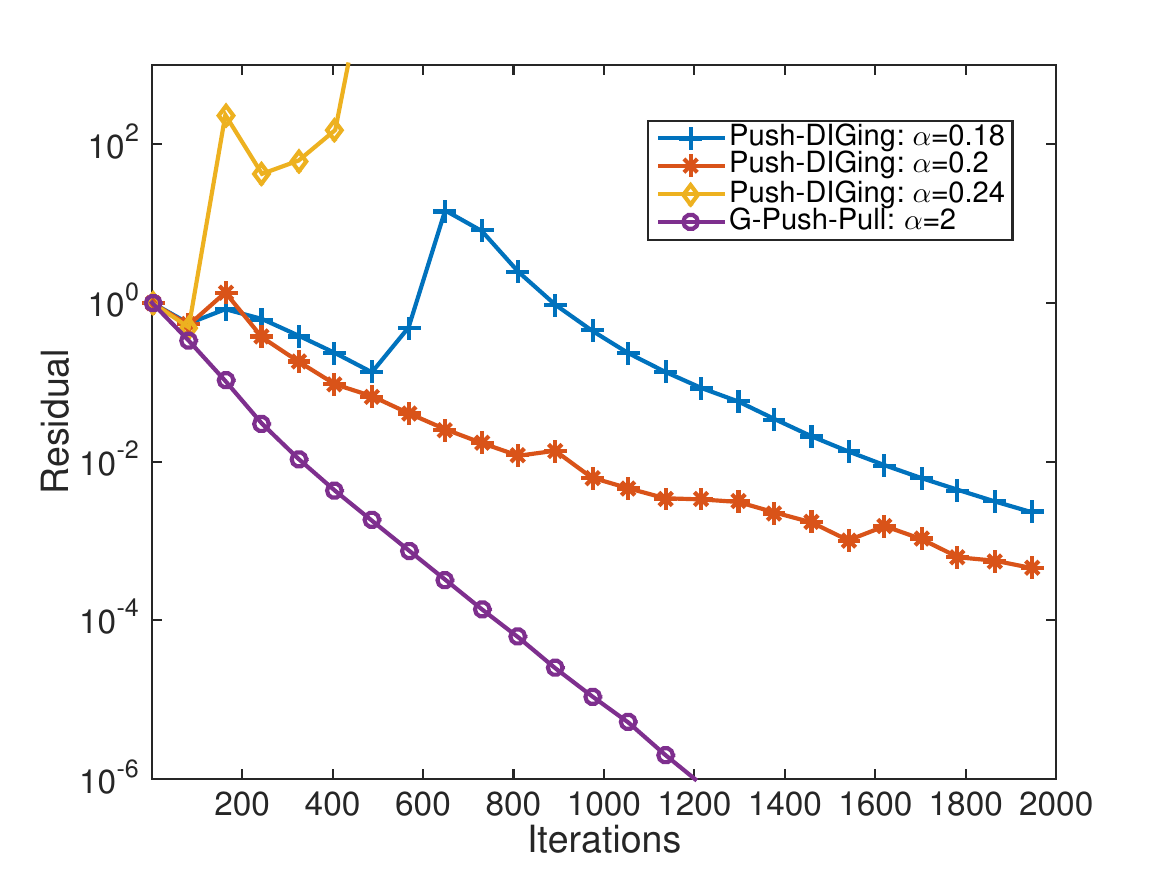}\label{fig:asynet}
			}
			\caption{Plots of (normalized) residuals versus the number of iterations over (a) a fixed directed network and (b) an asynchronous directed network. The simulation results for asynchronous stetting are averaged over 20 runs.}
		\end{center}
	\end{figure}

	\section{Conclusions}
	\label{sec: conclusion}
	In this paper, we have studied the problem of distributed optimization over a network.
	In particular, we proposed new distributed gradient-based methods (Push-Pull and G-Push-Pull) where each node maintains estimates of the optimal decision variable and the average gradient of the agents' objective functions. From the viewpoint of an agent, the information about the 
	gradients is pushed to its neighbors, while the information about the decision variable is pulled from its neighbors. The methods utilize two different graphs for the information exchange among agents and work for different types of distributed architecture, including decentralized, centralized, and semi-centralized  architecture. We have showed that the algorithms converge linearly for strongly convex and smooth objective functions over a directed network both for synchronous and asynchronous random-gossip updates. In the simulations, we have also demonstrated the effectiveness of the proposed algorithm as compared to the state-of-the-arts.

	
	%

	\appendices
	\section{Proofs for Push-Pull}	
	\subsection{Proof of Lemma \ref{lem: eigenvectors u v}}
	\label{subsec: proof eigenvectors u v}
	Denote by $u_i$ the $i$th element of $u$. We first prove that $u_i>0$ iff $i\in\mathcal{R}_\mathbf{R}$. Note that there exists an order of vertices such that $\mathbf{R}$ can be rewritten as
	$\tilde{\mathbf{R}}:=\begin{bmatrix}
	\mathbf{R}_1 & \mathbf{0}\\
	\mathbf{R}_2 & \mathbf{R}_3
	\end{bmatrix},$
	where $\mathbf{R}_1$ is a square matrix corresponding to vertices in $\mathcal{R}_\mathbf{R}$. Since the induced subgraph $\mathcal{G}_{\mathbf{R}_1}$ is strongly connected,
	$\mathbf{R}_1$ is row stochastic and irreducible. In light of the Perron-Frobenius theorem, $\mathbf{R}_1$ has a strictly positive left eigenvector $u_1^{\T}$ (with $u_1^{\T} \mathbf{1}=n$) corresponding to eigenvalue $1$. It follows that $[u_1, \mathbf{0}]^{\T}$ is a row eigenvector of $\tilde{\mathbf{R}}$, which is also unique from the Perron-Frobenius theorem. Since reordering of vertices does not change the corresponding eigenvector (up to permutation in the same oder of vertices), we have $u_i>0$ iff $i\in\mathcal{R}_\mathbf{R}$.
	Similarly, we can show that $v_j>0$ iff $j\in\mathcal{R}_{\mathbf{C}^{\T}}$. Since $\mathcal{R}_\mathbf{R}\cap\mathcal{R}_{\mathbf{C}^{\T}}\neq \emptyset$ from Assumption \ref{asp: nonempty root set}, we have $u^{\T}v>0$.
	
	\subsection{Proof of Lemma \ref{lem: Lipschitz implications}}
	\label{subsec: proof lemma Lipschitz implications}
	In light of Assumption \ref{asp; strconvex Lipschitz} and (\ref{oy_k}),
	$
	\|\oy_k-g_k\|_2=\frac{1}{n}\|\mathbf{1}^{\T}\nabla F(\mx_{k})-\mathbf{1}^{\T}\nabla F(\mathbf{1}\ox_k)\|_2
	\le \frac{L}{n}\sum_{i=1}^n \|x_{i,k}-\ox_k\|_2\le \frac{L}{\sqrt{n}}\|\mx_k-\mathbf{1}\ox_k\|_2,
	$
	and
	$
	\|g_k\|_2=\frac{1}{n}\|\mathbf{1}^{\T}\nabla F(\mathbf{1}\ox_k)-\mathbf{1}^{\T}\nabla F(\mathbf{1}x^*)\|_2
	\le \frac{L}{n}\sum_{i=1}^n \|\ox_k-x^*\|_2=L\|\ox_k-x^*\|_2.
	$
	Proof of the last relation can be found in \cite[Lemma 10]{qu2017harnessing}.
	
	\subsection{Proof of Lemma \ref{lem: spectral radii}}
	\label{subsec: proof lemma spectral radii}
	In light of \cite[Lemma 3.4]{ren2005consensus}, under Assumptions \ref{asp: stochastic}-\ref{asp: nonempty root set}, spectral radii of $\mR$ and $\mC$ are both equal to $1$ (the corresponding eigenvalues have multiplicity $1$). Suppose for some $\lambda, \tilde{u}\neq0$,
	$
	\tilde{u}^{\T}\left(\mR-\frac{\mathbf{1}u^{\T}}{n}\right)=\lambda\tilde{u}^{\T}.
	$
	Since $\mathbf{1}$ is a right eigenvector of $(\mR-\mathbf{1}u^{\T}/n)$ corresponding to eigenvalue $0$, $\tilde{u}^{\T}\mathbf{1}=0$ (see \cite[Theorem 1.4.7]{horn1990matrix}). We have
	$
	\tilde{u}^{\T} \mR=\lambda\tilde{u}.
	$
	Hence $\lambda$ is also an eigenvalue of $\mR$. Noticing that $u^{\T}\mathbf{1}=n$, we have $\tilde{u}^{\T}\neq u^{\T}$ so that $\lambda<1$. We conclude that $\sigma_{\mrR}<1$. Similarly we can obtain $\sigma_{\mrC}<1$.
	
	\subsection{Proof of Lemma \ref{lem: matrix norm production}}
	\label{subsec: proof lemma matrix norm production}
	By Definition \ref{def: norm n p}, 
	\begin{equation*}
	\begin{array}{l}
	\quad\|W\mx\|=\|[\|W\mx^1\|,\|W\mx^2\|,\ldots,\|W\mx^p\|]\|_2\\
	\le \|[\|W\|\|\mx^1\|,\|W\|\|\mx^2\|,\ldots,\|W\|\|\mx^p\|]\|_2\\
	=\|W\|\|[\|\mx^1\|,\|\mx^2\|,\ldots,\|\mx^p\|]\|_2=\|W\|\|\mx\|,
	\end{array}
	\end{equation*}
	and
	$
	\|wx\|=\|[\|wx^1\|,\|wx^2\|,\ldots,\|wx^p\|]\|_2=\|w\|\|[|x^1|,|x^2|,\ldots,|x^p|]\|_2=\|w\|\|x\|_2.
	$

	\subsection{Proof of Lemma \ref{lem: important inequalities}}
	\label{proof: important inequalities}
	The three inequalities embedded in (\ref{main ineqalities}) come from (\ref{ox pre}), (\ref{mx-ox pre}), and (\ref{my-oy pre}), respectively.
	First, by Lemma \ref{lem: Lipschitz implications} and Lemma \ref{lem: norm equivalence},
	we obtain from (\ref{ox pre}) that
	\begin{align*}
	\begin{array}{l}
	\quad\|\ox_{k+1}-x^*\|_2\le \|\ox_k-\alpha'g_k-x^*\|_2+\alpha'\|\oy_k-g_k\|_2\\
	\quad +\frac{1}{n}\|u^{\T}\balpha(\my_k-v\oy_k)\|_2\\
	\le (1-\alpha'\mu)\|\ox_k-x^*\|_2+\frac{\alpha'L}{\sqrt{n}}\|\mx_k-\mathbf{1}\ox_k\|_2\\
	\quad+\frac{\|u\|_2\|\balpha\|_2}{n}\|\my_k-v\oy_k\|_2\\
	\le (1-\alpha'\mu)\|\ox_k-x^*\|_2+\frac{\alpha'L}{\sqrt{n}}\|\mx_k-\mathbf{1}\ox_k\|_{\mrR}\\
	\quad+\frac{\hat{\alpha}\|u\|_2}{n}\|\my_k-v\oy_k\|_{\mrC}.
	\end{array}
	\end{align*}
	Second, by relation (\ref{mx-ox pre}), Lemma \ref{lem: matrix norm production} and Lemma \ref{lem: norm equivalence}, we see that
	\begin{equation*}
	\hspace{-0.6em}\begin{array}{l}
	\quad\|\mx_{k+1}-\mathbf{1}\ox_{k+1}\|_{\mrR} \le \sigma_{\mrR}\|\mx_k-\mathbf{1}\ox_k\|_{\mrR}+\sigma_{\mrR}\|\balpha\|_{\mrR}\|\my_k\|_{\mrR}\\
	\le \sigma_{\mrR}\|\mx_k-\mathbf{1}\ox_k\|_{\mrR}+\hat{\alpha}\sigma_{\mrR}\|\my_k-v\oy_k\|_{\mrR}+\hat{\alpha}\sigma_{\mrR}\|v\|_{\mrR}\|\oy_k\|_2\\
	\le \sigma_{\mrR}\|\mx_k-\mathbf{1}\ox_k\|_{\mrR}+\hat{\alpha}\sigma_{\mrR}\|\my_k-v\oy_k\|_{\mrR}\\
	\quad +\hat{\alpha}\sigma_{\mrR}\|v\|_{\mrR}\left(\frac{L}{\sqrt{n}}\|\mx_k-\mathbf{1}\ox_k\|_2+L\|\ox_k-x^*\|_2\right)\\
	\le \sigma_{\mrR}\left(1+\hat{\alpha} \|v\|_{\mrR} \frac{L}{\sqrt{n}}\right)\|\mx_k-\mathbf{1}\ox_k\|_{\mrR}+\hat{\alpha}\sigma_{\mrR}\delta_{\mrR,\mrC}\|\my_k-v\oy_k\|_{\mrC}\\
	\quad+\hat{\alpha}\sigma_{\mrR}\|v\|_{\mrR} L\|\ox_k-x^*\|_2.
	\end{array}
	\end{equation*}
	Lastly, it follows from (\ref{my-oy pre}),  Lemma \ref{lem: matrix norm production} and Lemma \ref{lem: norm equivalence} that
	\begin{equation*}
	\begin{array}{l}
	\quad\|\my_{k+1}-v\oy_{k+1}\|_{\mrC} \le \sigma_{\mrC}\|\my_k-v\oy_k\|_{\mrC}\\
	\quad+c_0 \delta_{\mrC,2}L\|\mx_{k+1}-\mx_k\|_2\\
	= \sigma_{\mrC}\|\my_k-v\oy_k\|_{\mrC}+c_0 \delta_{\mrC,2}L\|(\mR-\mI)(\mx_k-\mathbf{1}\ox_k)-\mR\balpha \my_k\|_2\\
	\le \sigma_{\mrC}\|\my_k-v\oy_k\|_{\mrC}+c_0 \delta_{\mrC,2}L\left(\|\mR-\mI\|_2\|\mx_k-\mathbf{1}\ox_k\|_2\right.\\
	\quad\left.+\left\|\mR\balpha(\my_k-v\oy_k)+\mR\balpha v\oy_k\right\|_2\right)\\
	\le \sigma_{\mrC}\|\my_k-v\oy_k\|_{\mrC}+c_0\delta_{\mrC,2}L(\|\mR-\mI\|_2\|\mx_k-\mathbf{1}\ox_k\|_2\\
	\quad+ \|\mR\|_2\|\balpha\|_2\|\my_k-v\oy_k\|_2+\|\mR\|_2 \|\balpha\|_2 \|v\|_2\|\oy_k\|_2).
	\end{array}
	\end{equation*}
	In light of Lemma \ref{lem: Lipschitz implications},
	\[
	\|\oy_k\|_2\le \left(\frac{L}{\sqrt{n}}\|\mx_k-\mathbf{1}\ox_k\|_2+L\|\ox_k-x^*\|_2\right).\]
	Hence
	\begin{equation*}
	\begin{array}{l}
	\quad\|\my_{k+1}-v\oy_{k+1}\|_{\mrC} \\
	\le \left(\sigma_{\mrC}+ \hat{\alpha}c_0\delta_{\mrC,2}\|\mR\|_2 L\right)\|\my_k-v\oy_k\|_{\mrC}\\
	\quad+c_0 \delta_{\mrC,2} L\|\mR-\mI\|_2\|\mx_k-\mathbf{1}\ox_k\|_2\\
	\quad+\hat{\alpha}c_0 \delta_{\mrC,2}\|\mR\|_2 \|v\|_2L\left(\frac{L}{\sqrt{n}}\|\mx_k-\mathbf{1}\ox_k\|_2+L\|\ox_k-x^*\|_2\right)\\
	\le \left(\sigma_{\mrC}+ \hat{\alpha}c_0\delta_{\mrC,2}\|\mR\|_2 L\right)\|\my_k-v\oy_k\|_{\mrC}\\
	\quad+c_0 \delta_{\mrC,2} L\left(\|\mR-\mI\|_2+\hat{\alpha}\|\mR\|_2 \|v\|_2\frac{L}{\sqrt{n}}\right)\|\mx_k-\mathbf{1}\ox_k\|_{\mrR}\\
	\quad+\hat{\alpha}c_0 \delta_{\mrC,2}\|\mR\|_2 \|v\|_2 L^2\|\ox_k-x^*\|_2.
	\end{array}
	\end{equation*}
	
	\section{Proofs for G-Push-Pull}
	\subsection{Proof of Lemma \ref{lem: x-xstar_pre gpp}}
	\label{proof: lem: x-xstar_pre gpp}
	By (\ref{algorithm G-P-P}),
	\begin{equation*}
	\begin{array}{l}
	\quad\bar{x}_{k+1}=\frac{\bar{u}^{\T}}{n}\left(\mR_k\mx_k-\alpha\mQ_k \my_k\right)\\
	=\bar{x}_k+\frac{\bar{u}^{\T}}{n}\left(\tilde{\mR}_k\mx_k-\alpha\mQ_k \my_k\right)\\
	=\bar{x}_k-\alpha\frac{\bar{u}^{\T}}{n}\mQ_k \my_k+\frac{\bar{u}^{\T}}{n}\tilde{\mR}_k\left(\mx_k-\mathbf{1}\ox_k\right).
	\end{array}
	\end{equation*}
	It follows that
	\begin{equation*}
	\begin{array}{l}
	\quad\|\bar{x}_{k+1}-x^*\|_2^2\\
	=\left\|\bar{x}_k-\alpha\frac{\bar{u}^{\T}}{n}\mQ_k \my_k+\frac{\bar{u}^{\T}}{n}\tilde{\mR}_k\left(\mx_k-\mathbf{1}\ox_k\right)-x^*\right\|_2^2\\
	= \|\bar{x}_k-x^*\|_2^2\\
	\quad-2\left\langle \bar{x}_k-x^*, \alpha\frac{\bar{u}^{\T}}{n}\mQ_k \my_k-\frac{\bar{u}^{\T}}{n}\tilde{\mR}_k\left(\mx_k-\mathbf{1}\ox_k\right)\right\rangle\\
	\quad+\left\|\alpha\frac{\bar{u}^{\T}}{n}\mQ_k \my_k-\frac{\bar{u}^{\T}}{n}\tilde{\mR}_k\left(\mx_k-\mathbf{1}\ox_k\right)\right\|_2^2.
	\end{array}
	\end{equation*}
	Taking conditional expectation on both sides,
	\begin{equation*}
	\begin{array}{l}
	\bE\left[\|\bar{x}_{k+1}-x^*\|_2^2\mid\mathcal{H}_k\right]
	= \|\bar{x}_k-x^*\|_2^2\\
	-2\alpha\left\langle \bar{x}_k-x^*, \frac{\bar{u}^{\T}}{n}\bE[\mQ_k] \my_k\right\rangle
	+\alpha^2\bE\left[\left\|\frac{\bar{u}^{\T}}{n}\mQ_k \my_k\right\|_2^2\mid\mathcal{H}_k\right]\\
	+\bE\left[\left\|\frac{\bar{u}^{\T}}{n}\tilde{\mR}_k\left(\mx_k-\mathbf{1}\ox_k\right)\right\|_2^2\mid\mathcal{H}_k\right].
	\end{array}
	\end{equation*}
	We now bound the last two terms on the right-hand side. First,
	\begin{equation*}
	\begin{array}{l}
	\quad\bE\left[\left\|\frac{\bar{u}^{\T}}{n}\mQ_k \my_k\right\|_2^2\mid\mathcal{H}_k\right]=\my_k^{\T}\bE\left[\mQ_k^{\T}\frac{\bar{u}}{n}\frac{\bar{u}^{\T}}{n}\mQ_k\right] \my_k\\
	=\my_k^{\T}\bE\left[\mQ_k^{\T}\frac{\bar{u}\bar{u}^{\T}}{n^2}\mQ_k\right] \my_k\le \frac{1}{n^2}\left\|\bE\left[\mQ_k^{\T}\bar{u}\bar{u}^{\T}\mQ_k\right]\right\|_2\|\my_k\|_2^2\\
	\le \frac{2}{n^2}\left\|\bE\left[\mQ_k^{\T}\bar{u}\bar{u}^{\T}\mQ_k\right]\right\|_2\left[2\|\my_k-\bar{v}\oy_k\|_2^2+2\|\bar{v}\|_2^2\|\oy_k\|_2^2\right].
	\end{array}
	\end{equation*}
	Second,
	\begin{equation*}
	\begin{array}{l}
	\quad\bE\left[\left\|\frac{\bar{u}^{\T}}{n}\tilde{\mR}_k\left(\mx_k-\mathbf{1}\ox_k\right)\right\|_2^2\mid\mathcal{H}_k\right]\\
	=\left(\mx_k-\mathbf{1}\ox_k\right)^{\T}\bE\left[\tilde{\mR}_k^{\T}\frac{\bar{u}\bar{u}^{\T}}{n^2}\tilde{\mR}_k\right]\left(\mx_k-\mathbf{1}\ox_k\right)\\
	\le \frac{1}{n^2}\left\|\bE\left[\tilde{\mR}_k^{\T}\bar{u}\bar{u}^{\T}\tilde{\mR}_k\right]\right\|_2\|\mx_k-\mathbf{1}\ox_k\|_2^2.
	\end{array}
	\end{equation*}
	The term with inner product can be rewritten in the following way:
	\begin{equation*}
	\begin{array}{l}
	\quad\left\langle \bar{x}_k-x^*, \frac{\bar{u}^{\T}}{n}\bE[\mQ_k] \my_k\right\rangle=\left\langle \bar{x}_k-x^*, \frac{\bar{u}^{\T}}{n}\bE[\mQ_k] \bar{v}\oy_k\right\rangle\\
	\quad+\left\langle \bar{x}_k-x^*, \frac{\bar{u}^{\T}}{n}\bE[\mQ_k] (\my_k-\bar{v}\oy_k)\right\rangle\\
	=\eta\left\langle \bar{x}_k-x^*, \oy_k\right\rangle+\left\langle \bar{x}_k-x^*, \frac{\bar{u}^{\T}}{n}\bE[\mQ_k] (\my_k-\bar{v}\oy_k)\right\rangle.
	\end{array}
	\end{equation*}
	We now have
	\begin{equation*}
	\begin{array}{l}
	\quad\bE\left[\|\bar{x}_{k+1}-x^*\|_2^2\mid\mathcal{H}_k\right]\\
	\le \|\bar{x}_k-x^*\|_2^2-2\alpha\eta\left\langle \bar{x}_k-x^*, \oy_k\right\rangle\\
	\quad-2\alpha\left\langle \bar{x}_k-x^*, \frac{\bar{u}^{\T}}{n}\bE[\mQ_k] (\my_k-\bar{v}\oy_k)\right\rangle\\
	\quad+\frac{2\alpha^2}{{n^2}}\left\|\bE\left[\mQ_k^{\T}\bar{u}\bar{u}^{\T}\mQ_k\right]\right\|_2\|\my_k-\bar{v}\oy_k\|_2^2\\
	\quad+\frac{2\alpha^2}{n^2}\left\|\bE\left[\mQ_k^{\T}\bar{u}\bar{u}^{\T}\mQ_k\right]\right\|_2\|\bar{v}\|_2^2\|\oy_k\|_2^2\\
	\quad+\frac{1}{n^2}\left\|\bE\left[\tilde{\mR}_k^{\T}\bar{u}\bar{u}^{\T}\tilde{\mR}_k\right]\right\|_2\|\mx_k-\mathbf{1}\ox_k\|_2^2.
	\end{array}
	\end{equation*}
	Notice that
	\begin{equation*}
	\begin{array}{l}
	\quad\|\bar{x}_k-x^*\|_2^2-2\alpha\eta\left\langle \bar{x}_k-x^*, \oy_k\right\rangle\\
	=\|\bar{x}_k-x^*\|_2^2-2\alpha\eta\left\langle \bar{x}_k-x^*, g_k\right\rangle-2\alpha\eta\left\langle \bar{x}_k-x^*, \oy_k-g_k\right\rangle\\
	\le (1-2\alpha\eta\mu)\|\bar{x}_k-x^*\|_2^2-2\alpha\eta\left\langle \bar{x}_k-x^*, \oy_k-g_k\right\rangle\\
	\le (1-2\alpha\eta\mu)\|\bar{x}_k-x^*\|_2^2+2\alpha\eta \|\bar{x}_k-x^*\|\|\oy_k-g_k\|,
	\end{array}
	\end{equation*}
	and
	\[\|\oy_k\|_2^2=\|\oy_k-g_k+g_k\|_2^2\le 2\|\oy_k-g_k\|_2^2+2\|g_k\|_2^2.\]
	We have
	\begin{equation*}
	\hspace{-0.6em}\begin{array}{l}
	\bE\left[\|\bar{x}_{k+1}-x^*\|_2^2\mid\mathcal{H}_k\right]
	\le (1-2\alpha\eta\mu)\|\bar{x}_k-x^*\|_2^2\\
	+2\alpha\eta \|\bar{x}_k-x^*\|_2\|\oy_k-g_k\|_2\\
	+\frac{2\alpha}{n}\| \bar{x}_k-x^*\|_2 \left\|\bar{u}^{\T}\bE[\mQ_k]\right\|_2 \|\my_k-\bar{v}\oy_k\|_2\\
	+\frac{2\alpha^2}{{n^2}}\left\|\bE\left[\mQ_k^{\T}\bar{u}\bar{u}^{\T}\mQ_k\right]\right\|_2\|\my_k-\bar{v}\oy_k\|_2^2\\
	+\frac{4\alpha^2}{n^2}\left\|\bE\left[\mQ_k^{\T}\bar{u}\bar{u}^{\T}\mQ_k\right]\right\|_2\|\bar{v}\|_2^2(\|\oy_k-g_k\|_2^2+\|g_k\|_2^2)\\
	+\frac{1}{n^2}\left\|\bE\left[\tilde{\mR}_k^{\T}\bar{u}\bar{u}^{\T}\tilde{\mR}_k\right]\right\|_2\|\mx_k-\mathbf{1}\ox_k\|_2^2.
	\end{array}
	\end{equation*}
	Note that
	\begin{align*}
	\begin{array}{l}
	\quad 2\alpha\eta \|\bar{x}_k-x^*\|_2\|\oy_k-g_k\|_2\\
	\le \alpha\eta\left(\frac{\mu}{2}\|\bar{x}_k-x^*\|_2^2+\frac{2}{\mu}\|\oy_k-g_k\|_2^2\right),
	\end{array}
	\end{align*}
	and
	\begin{equation*}
	\begin{array}{l}
	\quad\frac{2\alpha}{n}\| \bar{x}_k-x^*\|_2 \left\|\bar{u}^{\T}\bE[\mQ_k]\|_2\|\my_k-\bar{v}\oy_k\right\|_2\\
	\le \alpha\eta\left(\frac{\mu}{2}\|\bar{x}_k-x^*\|_2^2+\frac{2}{\eta^2\mu n^2}\left\|\bar{u}^{\T}\bE[\mQ_k]\right\|_2^2\|\my_k-\bar{v}\oy_k\|_2^2\right).
	\end{array}
	\end{equation*}
	From Lemma \ref{lem: Lipschitz implications} we obtain
	\begin{equation*}
	\begin{array}{l}
	\quad\bE\left[\|\bar{x}_{k+1}-x^*\|_2^2\mid\mathcal{H}_k\right]
	\le (1-\alpha\eta\mu)\|\bar{x}_k-x^*\|_2^2\\
	\quad+\frac{2\alpha\eta}{\mu}\|\oy_k-g_k\|_2^2+\frac{2\alpha}{\eta \mu n^2}\left\|\bar{u}^{\T}\bE[\mQ_k]\right\|_2^2\|\my_k-\bar{v}\oy_k\|_2^2\\
	\quad+\frac{2\alpha^2}{{n^2}}\left\|\bE\left[\mQ_k^{\T}\bar{u}\bar{u}^{\T}\mQ_k\right]\right\|_2\|\my_k-\bar{v}\oy_k\|_2^2\\
	\quad+\frac{4\alpha^2}{n^2}\left\|\bE\left[\mQ_k^{\T}\bar{u}\bar{u}^{\T}\mQ_k\right]\right\|_2\|\bar{v}\|_2^2(\|\oy_k-g_k\|_2^2+\|g_k\|_2^2)\\
	\quad+\frac{1}{n^2}\left\|\bE\left[\tilde{\mR}_k^{\T}\bar{u}\bar{u}^{\T}\tilde{\mR}_k\right]\right\|_2\|\mx_k-\mathbf{1}\ox_k\|_2^2\\
	\le \left[1-\alpha\eta\mu+\frac{4\alpha^2}{n^2}\left\|\bE\left[\mQ_k^{\T}\bar{u}\bar{u}^{\T}\mQ_k\right]\right\|_2\|\bar{v}\|_2^2L^2\right]\|\bar{x}_k-x^*\|_2^2\\
	\quad+\left(\frac{2\alpha\eta L^2}{\mu n}+\frac{4\alpha^2 L^2}{n^3}\left\|\bE\left[\mQ_k^{\T}\bar{u}\bar{u}^{\T}\mQ_k\right]\right\|_2\|\bar{v}\|_2^2\right.\\
	\quad\left.+\frac{1}{n^2}\left\|\bE\left[\tilde{\mR}_k^{\T}\bar{u}\bar{u}^{\T}\tilde{\mR}_k\right]\right\|_2\right)\|\mx_k-\mathbf{1}\ox_k\|_2^2\\
	\quad+\left(\frac{2\alpha}{\eta \mu n^2}\left\|\bar{u}^{\T}\bE[\mQ_k]\right\|_2^2+\frac{2\alpha^2}{{n^2}}\left\|\bE\left[\mQ_k^{\T}\bar{u}\bar{u}^{\T}\mQ_k\right]\right\|_2\right)\|\my_k-\bar{v}\oy_k\|_2^2.
	\end{array}
	\end{equation*}
	
	\subsection{Proof of Lemma \ref{lem: linear system gossip}}
	\label{proof: lem: linear system gossip}
	
	The first inequality follows directly from Lemma \ref{lem: x-xstar_pre gpp} and Lemma \ref{lem: norm equivalence 2}.
	
	For the second inequality, note that from (\ref{mx ox pre gpp}) we have
	\begin{equation*}
	\begin{array}{l}
	\quad\left\|\mx_{k+1}-\mathbf{1}\ox_{k+1}\right\|_{\mrS}^2\\
	\le \left[\left\|\left(\mI-\frac{\mathbf{1}\bar{u}^{\T}}{n}\right)\mR_k\mx_k\right\|_{\mrS}+\left\|\left(\mI-\frac{\mathbf{1}\bar{u}^{\T}}{n}\right)\alpha\mQ_k\my_k\right\|_{\mrS}\right]^2\\
	=\left\|\left(\mI-\frac{\mathbf{1}\bar{u}^{\T}}{n}\right)\mR_k\mx_k\right\|_{\mrS}^2+\alpha^2\left\|\left(\mI-\frac{\mathbf{1}\bar{u}^{\T}}{n}\right)\mQ_k\my_k\right\|_{\mrS}^2\\
	\quad+2\alpha\left\|\left(\mI-\frac{\mathbf{1}\bar{u}^{\T}}{n}\right)\mR_k\mx_k\right\|_{\mrS}\left\|\left(\mI-\frac{\mathbf{1}\bar{u}^{\T}}{n}\right)\mQ_k\my_k\right\|_{\mrS}.
	\end{array}
	\end{equation*}
	In light of Lemma \ref{lem: x_consensus gpp},
	\begin{equation*}
	\begin{array}{l}
	\quad\bE\left[\left\|\mx_{k+1}-\mathbf{1}\ox_{k+1}\right\|_{\mrS}^2\mid\mathcal{H}_k\right]\\
	\le \left(\sigma_{\bar{\mrR}}+\frac{1-\sigma_{\bar{\mrR}}}{2}\right)\left\|\mx_k-\mathbf{1}\ox_k\right\|_{\mrS}^2+\alpha^2\left(1+\frac{2\sigma_{\bar{\mrR}}}{1-\sigma_{\bar{\mrR}}}\right)\|\bar{\mV}_{\mrQ}\|_2\|\my_k\|_{\mrS}^2\\
	=\frac{(1+\sigma_{\bar{\mrR}})}{2}\left\|\mx_k-\mathbf{1}\ox_k\right\|_{\mrS}^2+\alpha^2\left(\frac{1+\sigma_{\bar{\mrR}}}{1-\sigma_{\bar{\mrR}}}\right)\|\bar{\mV}_{\mrQ}\|_2\|\my_k\|_{\mrS}^2.
	\end{array}
	\end{equation*}
	Since
	\begin{equation*}
	\begin{array}{l}
	\quad\|\my_k\|_{\mrS}^2= \left\|\my_k-\bar{v}\oy_k+\bar{v}(\oy_k-g_k)+\bar{v}g_k\right\|_{\mrS}^2\\
	\le 3\left[\|\my_k-\bar{v}\oy_k\|_{\mrS}^2+\|\bar{v}\|_{\mrS}^2\|\oy_k-g_k\|_2^2+\|\bar{v}\|_{\mrS}^2\|g_k\|_2^2\right]\\
	\le 3\left[\|\my_k-\bar{v}\oy_k\|_{\mrS}^2+\frac{L^2}{n}\|\bar{v}\|_{\mrS}^2\|\mx_k-\mathbf{1}\ox_k\|_2^2+L^2\|\bar{v}\|_{\mrS}^2\|\ox_k-x^*\|_2^2\right].
	\end{array}
	\end{equation*}
	We have
	\begin{equation*}
	\begin{array}{l}
	\quad\bE\left[\left\|\mx_{k+1}-\mathbf{1}\ox_{k+1}\right\|_{\mrS}^2\mid\mathcal{H}_k\right]\le \frac{(1+\sigma_{\bar{\mrR}})}{2}\left\|\mx_k-\mathbf{1}\ox_k\right\|_{\mrS}^2\\
	\quad+3\alpha^2\left(\frac{1+\sigma_{\bar{\mrR}}}{1-\sigma_{\bar{\mrR}}}\right)\|\bar{\mV}_{\mrQ}\|_2\left[\|\my_k-\bar{v}\oy_k\|_{\mrS}^2+\frac{L^2}{n}\|\bar{v}\|_{\mrS}^2\|\mx_k-\mathbf{1}\ox_k\|_2^2\right.\\
	\quad +L^2\|\bar{v}\|_{\mrS}^2\|\ox_k-x^*\|_2^2\Big]\\
	\le \left[\frac{(1+\sigma_{\bar{\mrR}})}{2}+\frac{3\alpha^2L^2}{n}\left(\frac{1+\sigma_{\bar{\mrR}}}{1-\sigma_{\bar{\mrR}}}\right)\|\bar{\mV}_{\mrQ}\|_2\|\bar{v}_{\mrS}\|^2\right]\left\|\mx_k-\mathbf{1}\ox_k\right\|_{\mrS}^2\\
	\quad+3\alpha^2L^2\left(\frac{1+\sigma_{\bar{\mrR}}}{1-\sigma_{\bar{\mrR}}}\right)\|\bar{\mV}_{\mrQ}\|_2\|\bar{v}\|_{\mrS}^2\|\ox_k-x^*\|_2^2\\
	\quad+3\alpha^2\left(\frac{1+\sigma_{\bar{\mrR}}}{1-\sigma_{\bar{\mrR}}}\right)\|\bar{\mV}_{\mrQ}\|_2\delta_{\mrS,\mrD}\|\my_k-\bar{v}\oy_k\|_{\mrD}^2.
	\end{array}
	\end{equation*}
	
	For the last inequality, from (\ref{my oy pre gpp}) and Lemma \ref{lem: y_alignment gpp} we have
	\begin{equation*}
	\begin{array}{l}
	\quad\left\|\my_{k+1}-\frac{\bar{v} \mathbf{1}^{\T}}{n}\my_{k+1}\right\|_{\mrD}^2\\
	\le \left(\left\|\left(\mI-\frac{\bar{v} \mathbf{1}^{\T}}{n}\right)\mC_k\my_k\right\|_{\mrD}\right.\\
	\quad\left.+\left\|\left(\mI-\frac{\bar{v} \mathbf{1}^{\T}}{n}\right)\left[\nabla F(\mx_{k+1})-\nabla F(\mx_k)\right]\right\|_{\mrD}\right)^2\\
	\le \left(1+\frac{1-\sigma_{\bar{\mrC}}}{2\sigma_{\bar{\mrC}}}\right)\left(\sigma_{\bar{\mrC}}\left\|\my_k-\bar{v}\oy_k\right\|_{\mrD}^2+2\gamma^2\|\bar{\mV}_{\mrE}\|_2\left\|\bar{v}\oy_k\right\|_{\mrD}^2\right)\\
	\quad+\left(1+\frac{2\sigma_{\bar{\mrC}}}{1-\sigma_{\bar{\mrC}}}\right)\left\|\left(\mI-\frac{\bar{v} \mathbf{1}^{\T}}{n}\right)\left[\nabla F(\mx_{k+1})-\nabla F(\mx_k)\right]\right\|_{\mrD}^2\\
	= \frac{(1+\sigma_{\bar{\mrC}})}{2}\left\|\my_k-\bar{v}\oy_k\right\|_{\mrD}^2+\frac{(1+\sigma_{\bar{\mrC}})}{\sigma_{\bar{\mrC}}}\gamma^2\|\bar{\mV}_{\mrE}\|_2\left\|\bar{v}\oy_k\right\|_{\mrD}^2\\
	\quad+\left(\frac{1+\sigma_{\bar{\mrC}}}{1-\sigma_{\bar{\mrC}}}\right)\left\|\left(\mI-\frac{\bar{v} \mathbf{1}^{\T}}{n}\right)\left[\nabla F(\mx_{k+1})-\nabla F(\mx_k)\right]\right\|_{\mrD}^2.
	\end{array}
	\end{equation*}
	Notice that by (\ref{algorithm G-P-P}) and Assumption \ref{asp; strconvex Lipschitz},
	\begin{equation*}
	\begin{array}{l}
	\quad\|\nabla F(\mx_{k+1})-\nabla F(\mx_k)\|_{\mrD}^2
	\le \delta_{\mrD,2}^2L^2\|\mx_{k+1}-\mx_k\|_2^2\\
	=\delta_{\mrD,2}^2 L^2\|\mR_k\mx_k-\alpha\mQ_k \my_k-\mx_k\|_2^2\\
	\le 2\delta_{\mrD,2}^2 L^2\|(\mR_k-\mI)(\mx_k-\mathbf{1}\ox_k)\|_2^2+2\alpha^2\delta_{\mrD,2}^2 L^2\|\mQ_k\|_2^2\|\my_k\|_2^2\\
	\le 2\delta_{\mrD,2}^2 L^2\|\mR_k-\mI\|_2^2\|\mx_k-\mathbf{1}\ox_k\|_2^2\\
	\quad+4\alpha^2\delta_{\mrD,2}^2 L^2\|\mQ_k\|_2^2(\|\my_k-\bar{v}\oy_k\|_2^2+\|\bar{v}\oy_k\|_2^2).
	\end{array}
	\end{equation*}
	We have
	\begin{equation*}
	\hspace{-0.6em}\begin{array}{l}
	\quad\left\|\my_{k+1}-\bar{v}\oy_{k+1}\right\|_{\mrD}^2
	\le
	\frac{(1+\sigma_{\bar{\mrC}})}{2}\left\|\my_k-\bar{v}\oy_k\right\|_{\mrD}^2\\
	\quad+\frac{(1+\sigma_{\bar{\mrC}})}{\sigma_{\bar{\mrC}}}\gamma^2\|\bar{\mV}_{\mrE}\|_2\left\|\bar{v}\oy_k\right\|_{\mrD}^2\\
	\quad+\delta_{\mrD,2}^2L^2\left(\frac{1+\sigma_{\bar{\mrC}}}{1-\sigma_{\bar{\mrC}}}\right)\left\|\mI-\frac{\bar{v} \mathbf{1}^{\T}}{n}\right\|_{\mrD}^2\left[2\|\mR_k-\mI\|_2^2\|\mx_k-\mathbf{1}\ox_k\|_2^2\right.\\
	\quad\left.+4\alpha^2\|\mQ_k\|_2^2(\|\my_k-\bar{v}\oy_k\|_2^2+\|\bar{v}\oy_k\|_2^2)\right]\\ 
	\le\left[\frac{(1+\sigma_{\bar{\mrC}})}{2}+4\alpha^2 \delta_{\mrD,2}^2 L^2\left(\frac{1+\sigma_{\bar{\mrC}}}{1-\sigma_{\bar{\mrC}}}\right)\left\|\mI-\frac{\bar{v} \mathbf{1}^{\T}}{n}\right\|_{\mrD}^2 \|\mQ_k\|_2^2\right]\\
	\quad\cdot\left\|\my_k-\bar{v}\oy_k\right\|_2^2\\
	\quad+2\delta_{\mrD,2}^2 L^2\left(\frac{1+\sigma_{\bar{\mrC}}}{1-\sigma_{\bar{\mrC}}}\right)\left\|\mI-\frac{\bar{v} \mathbf{1}^{\T}}{n}\right\|_{\mrD}^2 \|\mR_k-\mI\|_2^2\|\mx_k-\mathbf{1}\ox_k\|_2^2\\
	\quad+\left[\frac{(1+\sigma_{\bar{\mrC}})}{\sigma_{\bar{\mrC}}}\gamma^2\|\bar{\mV}_{\mrE}\|_2\|\bar{v}\|_{\mrD}^2\right.\\
	\quad\left.+4\alpha^2\delta_{\mrD,2}^2 L^2\left(\frac{1+\sigma_{\bar{\mrC}}}{1-\sigma_{\bar{\mrC}}}\right)\left\|\mI-\frac{\bar{v} \mathbf{1}^{\T}}{n}\right\|_{\mrD}^2\|\mQ_k\|_2^2\|\bar{v}\|_2^2\right]\left\|\oy_k\right\|_2^2\\
	\le\left[\frac{(1+\sigma_{\bar{\mrC}})}{2}+4\alpha^2 \delta_{\mrD,2}^2 L^2\left(\frac{1+\sigma_{\bar{\mrC}}}{1-\sigma_{\bar{\mrC}}}\right)\left\|\mI-\frac{\bar{v} \mathbf{1}^{\T}}{n}\right\|_{\mrD}^2 \|\mQ_k\|_2^2\right]\\
	\quad\cdot\left\|\my_k-\bar{v}\oy_k\right\|_{\mrD}^2\\
	\quad+2\delta_{\mrD,2}^2 L^2\left(\frac{1+\sigma_{\bar{\mrC}}}{1-\sigma_{\bar{\mrC}}}\right)\left\|\mI-\frac{\bar{v} \mathbf{1}^{\T}}{n}\right\|_{\mrD}^2 \|\mR_k-\mI\|_2^2\|\mx_k-\mathbf{1}\ox_k\|_{\mrS}^2\\
	\quad+2L^2\left[\frac{(1+\sigma_{\bar{\mrC}})}{\sigma_{\bar{\mrC}}}\gamma^2\|\bar{\mV}_{\mrE}\|_2\|\bar{v}\|_{\mrD}^2\right.\\
	\quad\left.+4\alpha^2\delta_{\mrD,2}^2 L^2\left(\frac{1+\sigma_{\bar{\mrC}}}{1-\sigma_{\bar{\mrC}}}\right)\left\|\mI-\frac{\bar{v} \mathbf{1}^{\T}}{n}\right\|_{\mrD}^2\|\mQ_k\|_2^2\|\bar{v}\|_2^2\right]\\
	\quad \cdot \left(\frac{1}{n}\|\mx_k-\mathbf{1}\ox_k\|_2^2+\|\ox_k-x^*\|_2^2\right).
	\end{array}
	\end{equation*}
	
	%
	%

	\ifCLASSOPTIONcaptionsoff
	\newpage
	\fi

	
	
	\bibliographystyle{IEEEtran}
	\bibliography{mybib_pushpull}
	%
	%
	%
	
	%
	
	\newpage
	
	\begin{IEEEbiography}[{\includegraphics[width=1in,height=1.2in,clip,keepaspectratio]{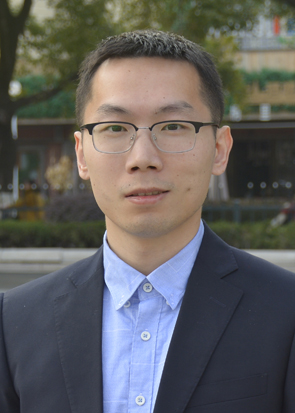}}]{Shi Pu} is currently an assistant professor in the Institute for Data and Decision Analytics, The Chinese University of Hong Kong, Shenzhen, China. He received a B.S. Degree in Engineering Mechanics from Peking University, in 2012, and a Ph.D. Degree in Systems Engineering from the University of Virginia, in 2016. He was a postdoctoral associate at the University of Florida, from 2016 to 2017, a postdoctoral scholar at Arizona State University, from 2017 to 2018, and a postdoctoral associate at Boston University, from 2018 to 2019. His research interests include distributed optimization, network science, machine learning, and game theory.
	\end{IEEEbiography}

\vspace{-1.5cm}
	
	\begin{IEEEbiography}[{\includegraphics[width=1in,height=1.2in,clip,keepaspectratio]{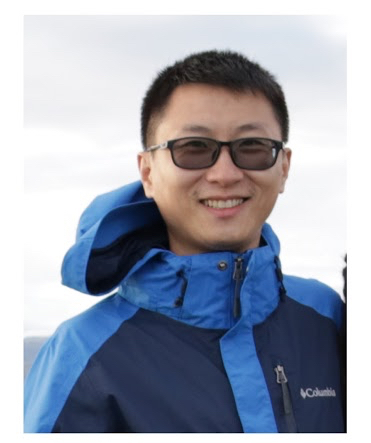}}]{Wei (Wilbur) Shi} was a postdoc in the Electrical and Computer Engineering Department of Princeton University, Princeton, NJ, USA. He obtained his Ph.D. in Control Science and Engineering from the University of Science and Technology of China, Hefei, Anhui, China. He was a postdoc in the Coordinated Science Laboratory, University of Illinois at Urbana-Champaign, Urbana, IL, USA. His current research interest distributes in optimization, learning, and control, and applications in cyber physical systems and internet of things. He was a recipient of the Young Author Best Paper Award in IEEE Signal Processing Society.
	\end{IEEEbiography}

\vspace{-1.5cm}
	
	\begin{IEEEbiography}[{\includegraphics[width=1in,height=1.2in,clip,keepaspectratio]{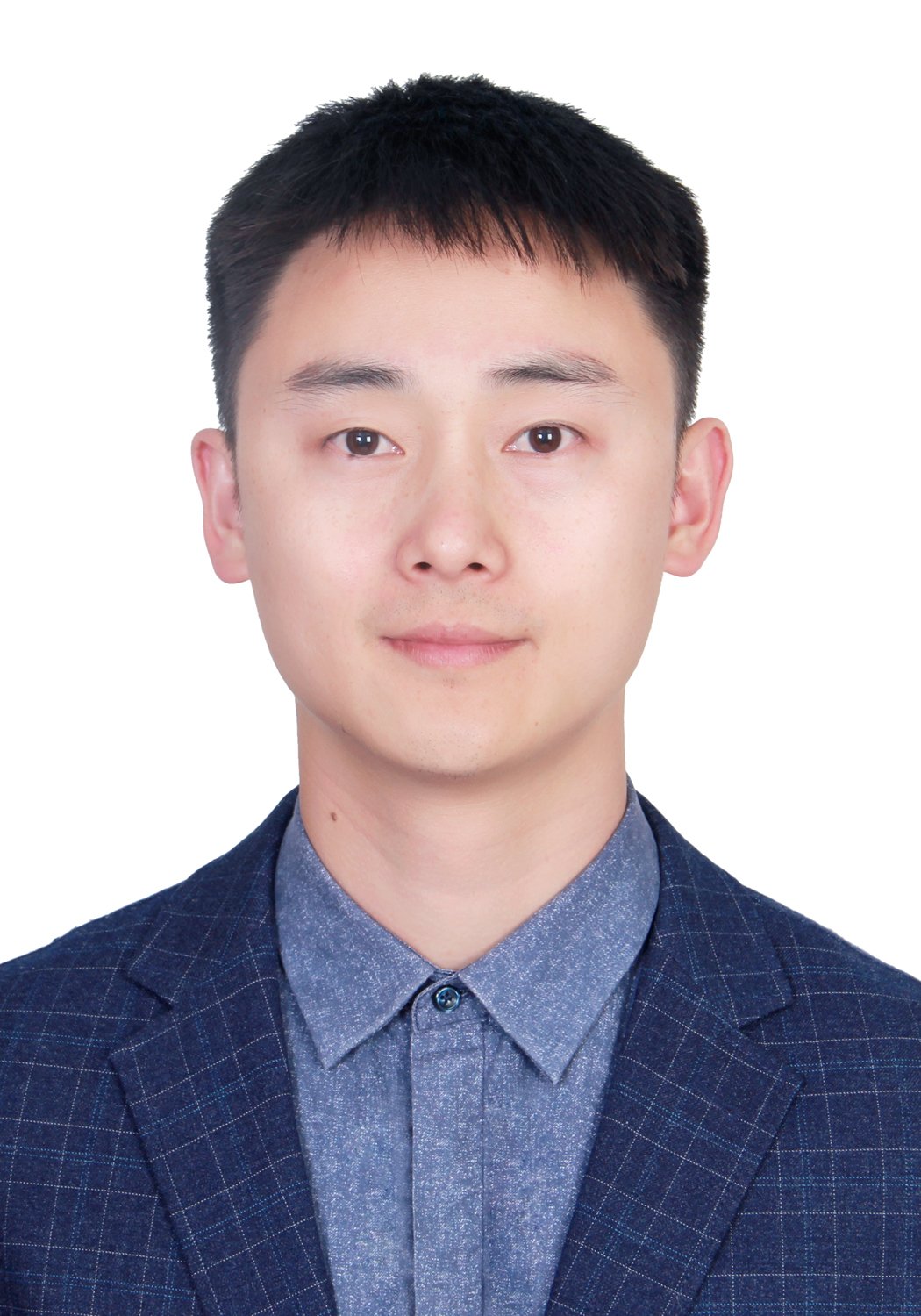}}]{Jinming Xu} received the B.S. degree in mechanical engineering from Shandong University, China, in 2009 and the Ph.D. degree in Electrical and Electronic Engineering from Nanyang Technological University (NTU), Singapore, in 2016. He was a research fellow of the EXQUITUS center at NTU from 2016 to 2017; he also received postdoctoral training in the Ira A. Fulton Schools of Engineering, Arizona State University, from 2017 to 2018, and School of Industrial Engineering, Purdue University, from 2018 to 2019, respectively. Currently, he is an assistant professor with the College of Control Science and Engineering at Zhejiang University, China. His research interests include distributed optimization and control, machine learning and network science.
		
	\end{IEEEbiography}
	
	\vspace{-1.5cm}
	
	\begin{IEEEbiography}[{\includegraphics[width=1in,height=1.2in,clip,keepaspectratio]{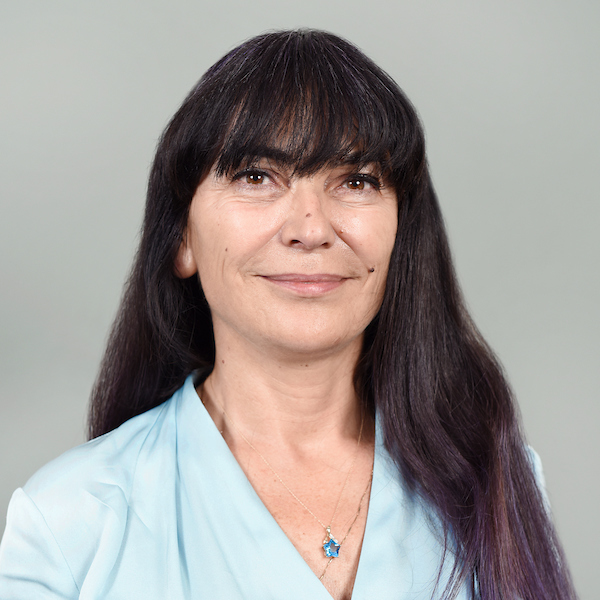}}]{Angelia Nedi\'{c}} has a Ph.D. from Moscow State University, Moscow, Russia, in Computational Mathematics and Mathematical Physics (1994), and a Ph.D. from Massachusetts Institute of Technology, Cambridge, USA in Electrical and Computer Science Engineering (2002). She has worked as a senior engineer in BAE Systems North America, Advanced Information Technology Division at Burlington, MA. Currently, she is a faculty member of the school of Electrical, Computer and Energy Engineering at Arizona State University at Tempe. Prior to joining Arizona State University, she has been a Willard Scholar faculty member at the University of Illinois at Urbana-Champaign. She is a recipient (jointly with her co-authors) of the Best Paper Awards at the Winter Simulation Conference 2013 and at the International Symposium on Modeling and Optimization in Mobile, Ad Hoc and Wireless Networks (WiOpt) 2015.  Her general research interest is in optimization, large scale complex systems dynamics, variational inequalities and games.
	\end{IEEEbiography}
	
	
	

\end{document}